\documentclass[a4paper,10pt,reqno]{amsart}
\usepackage[utf8]{inputenc}
\usepackage{amsmath}
\usepackage{cases}
\usepackage{amsfonts}
\usepackage[colorlinks,linkcolor=blue,citecolor=blue]{hyperref}
\usepackage{latexsym, amssymb, amsmath, amsthm, bbm}
\usepackage[all]{xy}
\usepackage{pgfplots}
\usepackage{enumerate}

\DeclareSymbolFont{EulerExtension}{U}{euex}{m}{n}
\DeclareMathSymbol{\euintop}{\mathop} {EulerExtension}{"52}
\DeclareMathSymbol{\euointop}{\mathop} {EulerExtension}{"48}

\allowdisplaybreaks[4]

\def \id{\operatorname{id}}

\def \C{\mathcal{C}}
\def \k{\Bbbk}

\def \dim{\operatorname{dim}}

\def \C{\mathcal{C}}

\def \span{\operatorname{span}}
\def \A{\mathcal{A}}
\def \B{\mathcal{B}}
\def \C{\mathcal{C}}

\def \E{\mathcal{E}}

\def \G{\mathcal{G}}

\usepackage{mathtools}

\numberwithin{equation}{section}

\newtheorem{theorem}{Theorem}[section]
\newtheorem{lemma}[theorem]{Lemma}
\newtheorem{proposition}[theorem]{Proposition}
\newtheorem{corollary}[theorem]{Corollary}
\newtheorem{definition}[theorem]{Definition}
\newtheorem{example}[theorem]{Example}
\newtheorem{remark}[theorem]{Remark}

\begin{document}
\title[Cosemisimple Hopf algebras]{A class of (infinite-dimensional) cosemisimple Hopf algebras constructed via abelian extensions}

\author[J. Yu]{Jing Yu$^\dag$}
\author[G. Liu]{Gongxiang Liu}
\author[K. Zhou]{Kun Zhou}
\author[X. Zhen]{Xiangjun Zhen}
\address{School of Mathematics, Nanjing University, Nanjing 210093, China}
\email{dg21210018@smail.nju.edu.cn}
\address{School of Mathematics, Nanjing University, Nanjing 210093, China}
\email{gxliu@nju.edu.cn}
\address{Beijing Institute of Mathematical Sciences and Applications, Beijing 101408, China}
\email{kzhou@bimsa.cn}
\address{School of Mathematics, Nanjing University, Nanjing 210093, China}
\email{xjzhen@smail.nju.edu.cn}
\thanks{$^\dag$ Corresponding author}
\date{}

\begin{abstract} In this paper, we aim to study abelian extensions for some infinite group. We show that the Hopf algebra $\Bbbk^G{}^\tau\#_{\sigma}\Bbbk F$ constructed through abelian extensions of $\Bbbk F$ by $\Bbbk^G$ for some (infinite) group $F$ and finite group $G$ is cosemisimple, and discuss when it admits a compact quantum group structure if $\Bbbk$ is the field of complex numbers $\mathbb{C}.$
We also find all the simple $\Bbbk^G{}^\tau\#_{\sigma}\Bbbk F$-comodules and attempt to determine the Grothendieck ring of the category of finite-dimensional right $\Bbbk^G{}^\tau\#_{\sigma}\Bbbk F$-comodules. Moreover, some new properties are given and some new examples are constructed.
\end{abstract}
\maketitle
\section{Introduction}
The concept of ``extension" has been recognized as important tools in order to understand the structure of the groups and of their representations since the work of Frobenius, Schur and others. This line of thought has been carried to the theory of Hopf algebras. See, for example, \cite{AD95, And96, Hof94, Kac68, Mas02, Sch93, Sin72, Tak81}. In particular, the Hopf algebra extensions
\begin{eqnarray*} K\xrightarrow{\iota} H \xrightarrow{\pi} A
\end{eqnarray*}
which are abelian in the sense that $A$ is cocommutative and $K$ is commutative have been applied especially for the classification problem of semisimple Hopf algebras. See, for example, \cite{DW13, GW00, Kas00, Mas95, Nat99}. Let $G, F$ be finite groups and $\k^G$ denote the dual Hopf algebra of $\k G$.
Abelian extensions
\begin{eqnarray}\label{eq:intro}
 \k^G\xrightarrow{\iota} H\xrightarrow{\pi} \k F
\end{eqnarray} of $\k F$ by $\k^G$
were classified by Masuoka (\cite[Proposition 1.5]{Mas02}), and the above $H$ can be expressed as $\k^G\#_{\sigma, \tau}\k F$.

At the same time, cosemisimple Hopf algebras or compact quantum groups have received considerable attention. See, for instance, \cite{AC13, AFH12, Chi14, CWW19, CK16, NY18}.
As is well-known, an extension of finite-dimensional semisimple Hopf algebras is semisimple (by \cite[Theorem 3.5]{MD92} and \cite[Proposition 3.1.18]{And96}).
Andruskiewitsch and Cuadra extended this fact to cosemisimple Hopf algebras under some mild conditions (\cite[Theorem 2.13]{AC13}).
A natural question arises:
can we remove the finiteness condition on $F$ in (\ref{eq:intro}) to construct a class of infinite-dimensional cosemisimple Hopf algebras through abelian extensions of $\k F$ by $\k^G$? Furthermore, we also aim to determine whether it admits a compact quantum group structure if $\k$ is the field of complex numbers $\mathbb{C}$.

Besides, much effort was put in the study of fusion categories from a Hopf algebraic perspective in recent years. The module categories of semisimple Hopf algebras
have been studied intensively by many authors. See, for example,  \cite{Bur17, JM09, KMM02, LM00, Orl24}.
It is interesting to generalize some of their results to the setting of semisimple tensor categories with infinitely many isomorphism classes of simple objects.
Note that the category of finite-dimensional right comodules over an infinite-dimensional cosemisimple Hopf algebra is precisely such a category.

The aim of this paper is to construct a class of (infinite-dimensional) cosemisimple Hopf algebras through abelian extensions, and study their right comodule categories.
For any matched pair of groups $(F, G)$ where $G$ is a finite group, we prove that $\k^G{}^\tau\#_{\sigma}\k F$ is a cosemisimple Hopf algebra with the described algebra and coalgebra structures (see Propositions \ref{prop:Hopf} and \ref{prop:cosemisimple}). In particular, if $\k$ is the field of complex numbers $\mathbb{C}$, we discuss when $\mathbb{C}^G{}^\tau\#_{\sigma}\mathbb{C} F$ admits a compact quantum group structure (see Proposition \ref{prop:CQG}).
Subsequently, we provide a characterization of their comodule catgeories. We give a description for simple right $\k^G{}^\tau\#_{\sigma}\k F$-comodules, and show that all the simple right $\k^G{}^\tau\#_{\sigma}\k F$-comodules may be described as induced comodules from some certain coalgebras. Our main result is Theorem \ref{thm:simple}, stating that:
\begin{theorem}\label{thm:intro}
Let $H=\k^G{}^\tau\#_{\sigma}\k F$ be the Hopf algebra as in Proposition \ref{prop:Hopf}, and fix an element $f\in F.$ For each $G$-orbit $O_f$ of $F$, let $V$ be a right $\k_{\tau_f}^{G_f}$-comodule. Let $\tilde{V}=(V\otimes \k f)\Box_{\k^{G_f}{}^\tau\#\k F} H$ be the induced right $H$-comodule as in Lemma \ref{lem:inducedcomodule}.
\begin{itemize}
  \item [(1)]If $V$ is a simple right $\k_{\tau_f}^{G_f}$-comodule, then $\tilde{V}$ is a simple right $H$-comodule.
  \item [(2)]Every simple right $H$-comodule is isomorphic to $\tilde{V}$ for some simple right $\k_{\tau_f}^{G_f}$-comodule, where $f$ ranges over a choice of one element in each $G$-orbit of $F$.
\end{itemize}
\end{theorem}

In fact, studying (co)semisimple Hopf algebras through the characters of (co)modules was shown to be a very effective way, which is due to the work of Larson \cite{Lar71}. See, for example, \cite{Bur15, Bur17, CW14, YLL24}. One of key points of the character theory is that one can identify a simple comodule with its character.
For our purpose, we characterize the characters of all the simple right $\k^G{}^\tau\#_{\sigma}\k F$-comodules.
This appears as Proposition \ref{prop:irreduciblecharacter} in this paper:
\begin{proposition}
With the notations in Lemma \ref{lem:inducedcomodule}, let $V=\span\{v^{(1)}, \cdots, v^{(m)}\}$ be a simple right $\k_{\tau_f}^{G_f}$-comodule and $\tilde{V}=(V\otimes \k f)\Box_{\k^{G_f}{}^\tau\#\k F} H$ be the induced right $H$-comodule. Suppose for any $1\leq i\leq m$, we have $$\rho(v^{(i)})=\sum_{j=1}^m v^{(j)}\otimes (\sum_{g\in G_f}a_{ji}^g p_g),$$
where $a_{ij}^g\in \k$ for any $1\leq i, j\leq m$ and $g\in G_f.$ Then the irreducible character $\chi(\tilde{V})$ of $\tilde{V}$ is
$$\sum\limits_{i=1}^m\sum\limits_{z\in T_f}\sum\limits_{g\in G_f}\tau(z^{-1},g^{-1};f)^{-1}\tau(z^{-1}g^{-1}z, z^{-1};f)a_{ii}^{g^{-1}}p_{z^{-1}g^{-1}z}\#(z^{-1}\triangleright f).$$
\end{proposition}

We should remark that the Grothendieck ring of the category of finite-dimensional comodules over a cosemisimple Hopf algebra is determined by the multiplication and dual of the irreducible characters (see Lemma \ref{lemma:char=ring}). With the help of the preceding proposition, we attempt to obtain the Grothendieck ring of the category of finite-dimensional right $\k^G{}^\tau\#_{\sigma}\k F$-comodules (see Propositions \ref{prop:tensorproduct} and \ref{prop:dual}). In particular, we discuss when a simple right $\k^G{}^\tau\#_{\sigma}\k F$-comodule is self-dual, and describe the Frobenius-Schur indicators for irreducible characters of all the simple right $\k^G{}^\tau\#_{\sigma}\k F$-comodules (see Remark \ref{prop:indicator}).
We finally apply these conclusions to the special case $\k^G{}\#\k F$ (see Lemma \ref{coro:smashsimple}, Propositions \ref{prop:smashdual} and \ref{prop:multiplication}) and construct some new examples.

The organization of this paper is as follows: Some definitions, notations and results related to characters of comodules, Grothendieck ring and matched pair of groups are presented in Section \ref{section2}.
In Section \ref{section3}, for any matched pair of groups $(F, G)$ where $G$ is a finite group, we prove that $\k^G{}^\tau\#_{\sigma}\k F$ is a cosemisimple Hopf algebra with the described algebra and coalgebra structures. Besides, we discuss when $\mathbb{C}^G{}^\tau\#_{\sigma}\mathbb{C} F$ is a compact quantum group.
We devote Section \ref{section3} to give a description for right $\k^G{}^\tau\#_{\sigma}\k F$-comodules, and show that all the simple right $\k^G{}^\tau\#_{\sigma}\k F$-comodules may be described as induced comodules from some certain coalgebras. In Section \ref{section4}, we characterize the characters of all the simple right $\k^G{}^\tau\#_{\sigma}\k F$-comodules. Then we attempt to determine the Grothendieck ring of the category of finite-dimensional right $\k^G{}^\tau\#_{\sigma}\k F$-comodules. Finally in Section \ref{section6}, we apply these conclusions to the special case $\k^G{}\#\k F$ and construct some new examples.

\section{Preliminaries}\label{section2}
Throughout this paper $\k$ denotes an \textit{algebraically closed field of characteristic} $0$ and all spaces are over $\k$. The tensor product over $\k$ is denoted simply by $\otimes$. We refer to \cite{Mon93} for the basics about Hopf algebras.
\subsection{Characters of comodules}
In this subsection, let $(H, \Delta, \varepsilon)$ be a coalgebra over $\k$.

Let us first recall the definition of multiplicative matrices.
\begin{definition}\emph{(}\cite[Definition 2.3]{Li22}\emph{)}
Let $(H,\Delta,\varepsilon)$ be a coalgebra over $\k$.
\begin{itemize}
  \item[(1)] A square matrix $\G=(g_{ij})_{r\times r}$ over $H$ is said to be multiplicative, if for any $1\leq i,j \leq r$, we have $\Delta(g_{ij})=\sum\limits_{t=1}^r g_{it}\otimes g_{tj}$ and $\varepsilon(g_{ij})=\delta_{i, j}$, where $\delta_{i, j}$ denotes the Kronecker notation;
  \item[(2)] A multiplicative matrix $\C$ is said to be basic, if its entries are linearly independent.
\end{itemize}
\end{definition}
Multiplicative matrices over a coalgebra $H$ can be understood as a generalization of group-like elements. We know that all the entries of a basic multiplicative matrix $\C$ span a simple subcoalgebra $C$ of $H$. Conversely, for any simple coalgebra $C$ over
$\k$, there exists a basic multiplicative matrix $\C$ whose entries span $C$ (for details, see \cite{LZ19}, \cite{Li22}).
And according to \cite[Lemma 2.4]{Li22}, the basic multiplicative matrix of the simple coalgebra $C$ would be unique up to the similarity relation.

\begin{definition}\emph{(}\cite[Definition 3.2.10]{Rad12}\emph{)}
Let $H$ be a coalgebra over $\k$ and $M$ be a right $H$-comodule with comodule structure $\rho: M\rightarrow M\otimes H$. The coefficient coalgebra $\operatorname{cf}(M)$ of $M$ is the smallest subcoalgebra of $H$ such that $\rho(M)\subseteq M\otimes \operatorname{cf}(M)$.
\end{definition}

Suppose that $M$ is a simple right $H$-comodule. According to \cite[Theorem 3.2.11]{Rad12}, $\operatorname{cf}(M)$ is a simple subcoalgebra of $H$.
In addition, it follows from \cite[Lemma 2.1]{Lar71} that for each basis $\{m_1, m_2, \cdots, m_n\}$ of $M$, there exists a basis $\{c_{ij}\mid 1\leq i, j\leq n\}$ of $\operatorname{cf}(M)$ such that
$$\rho(m_i)=\sum\limits_{j=1}^n m_j\otimes c_{ji},$$
where $(c_{ij})_{n\times n}$ is a basic multiplicative matrix of $\operatorname{cf}(M)$.

\begin{definition}\emph{(}\cite[Section 1]{Lar71}\emph{)}
Let $H$ be a coalgebra over $\k$ and $(M, \rho)$ be a right $H$-comodule. Pick a basis $\{m_1, m_2, \cdots, m_n\}$ for $M$ and let $\{m_1^*, m_2^*, \cdots, m_n^*\}$ be the basis of $W^*$ dual to this basis. The character $\chi(M)$ of the comodule $M$ is defined by
$$
\chi(M)=\sum_{i=1}^n (m_i^*\otimes \id)\rho(m_i).
$$
If $M$ is a simple right $H$-comodule, then $\chi(M)$ is called an irreducible character.
\end{definition}
Note that $\chi(M)$ does not depend on the particular choice of basis $\{m_1, m_2, \cdots, m_n\}.$
Moreover, for any finite-dimensional right $H$-comodules $M$ and $N$, we have $$\chi(M\bigoplus N)=\chi(M)+\chi(N).$$

The following lemma is immediate.
\begin{lemma}\emph{(}\cite[Section 1]{Lar71}\emph{)}\label{lem:char}
 Let $H$ be a coalgebra over $\k$ and $(M, \rho)$ be a simple right $H$-comodule. If $$\rho(m_i)=\sum_{j=1}^n m_j\otimes c_{ji},$$ for any $1 \leq i\leq n$, then $\chi(M)=\sum\limits_{i=1}^n c_{ii}.$
\end{lemma}
Let $M$ be a simple right $H$-comodule and $\C (M)$ be a basic multiplicative matrix of the simple subcoalgebra $\operatorname{cf}(M)$. It is apparent from Lemma \ref{lem:char} that $$\chi(M)=\operatorname{tr}(\C (M)).$$ For any simple right $H$-comodules $M, N$, we can easily prove that $$M\cong N$$ as $H$-comodules if and only if $$\operatorname{cf}(M)\cong \operatorname{cf}(N)$$ as coalgebras, if and only if $$\chi(M)=\chi(N).$$

\subsection{Grothendieck ring}
Let $\mathbb{Z}_+$ be the set of nonnegative integers. Some relevant concepts are recalled as follows.
\begin{definition}\emph{(}\cite[Definitions 2.1 and 2.2]{Ost03}\emph{)}
Let $A$ be an associative ring with unit which is free as a $\mathbb{Z}$-module.
\begin{itemize}
  \item[(1)]A $\mathbb{Z}_+$-basis of $A$ is a basis $B=\{b_{i}\}_{i\in I}$ such that $b_ib_j=\sum_{t\in I}c_{ij}^tb_t$, where $c_{ij}^t\in\mathbb{Z}_+$.
  \item[(2)]A ring with a fixed $\mathbb{Z}_+$-basis $\{b_i\}_{i\in I}$ is called a unital based ring if the following conditions hold:
  \begin{itemize}
  \item[(i)]$1$ is a basis element.
  \item[(ii)]Let $\tau: A\rightarrow \mathbb{Z}$ denote the group homomorphism defined by
  $$\tau(b_i)=\left\{
\begin{aligned}
1,~~~  \text{if} ~~~ b_i=1, \\
0,~~~  \text{if} ~~~ b_i\neq1.
\end{aligned}
\right.$$
There exists an involution $i \mapsto i^*$ of $I$ such that the induced map
$$a=\sum\limits_{i\in I}a_ib_i \mapsto a^*=\sum\limits_{i\in I}a_ib_{i^*},\;\; a_i\in \mathbb{Z}$$ is an anti-involution of $A$, and
$$\tau(b_ib_j)=\left\{
\begin{aligned}
1,~~~  \text{if} ~~~ i=j^*, \\
0,~~~  \text{if} ~~~ i\neq j^*.
\end{aligned}
\right.$$
  \end{itemize}
\end{itemize}
\end{definition}
In the following part, let $H$ be a cosemisimple Hopf algebra over $\k$ and $\mathcal{S}$ be the set of all the simple subcoalgebras of $H$.

For any matrix $\A=(a_{ij})_{r\times s}$ and $\B=(b_{ij})_{u\times v}$ over $H$, define $\A\odot^\prime \B$ as follow:
 $$
\A\odot^\prime \B=\left(\begin{array}{cccc}
      \A b_{11} &   \cdots &  \A b_{1v} \\
      \vdots &  \ddots & \vdots  \\
      \A b_{u1} &  \cdots & \A b_{uv}
    \end{array}\right).$$

For any $B, C\in\mathcal{S}$ with basic multiplicative matrices $\B, \C$ respectively.
It follows from
\cite[Proposition 2.6(2)]{Li22} that there exists an invertible matrix $L$ over $\k$ such that
\begin{equation}\label{equationCD}
L
(\B\odot^{\prime}\C) L^{-1}=
\left(\begin{array}{cccc}
      \E_1 & 0 & \cdots & 0  \\
      0 & \E_2 & \cdots & 0  \\
      \vdots & \vdots & \ddots & \vdots  \\
      0 & 0 & \cdots & \E_t
    \end{array}\right),
    \end{equation}
where $\E_1, \E_2, \cdots, \E_t$ are basic multiplicative matrices over $H$.
Define a multiplication on $\mathbb{Z}\mathcal{S}$ as follow: for $B, C\in \mathcal{S}$,
$$B\cdot C=\sum\limits_{i=1}^t E_i,$$
where $E_1, \cdots, E_t\in\mathcal{S}$ are well-defined with basic multiplicative matrices $\E_i$ as in (\ref{equationCD}).
It should be pointed out that $$\operatorname{tr}(\B)\operatorname{tr}(\C)=\sum\limits_{i=1}^t\operatorname{tr}(\E_i).$$

Let $S$ be the antipode of $H$. According to \cite[Theorem 3.3]{Lar71}, we know that the antipode $S$ of $H$ is bijective, and for each simple subcoalgebra $C$ of $H$, we have $S^2(C)=C.$ Thus we get an anti-involution $C\mapsto S(C)$ of $\mathcal{S}$.

With the multiplication and anti-involution defined above, we obtain the following Lemma.
\begin{lemma}\emph{(}\cite[Proposition 4.3]{YLL24}\emph{)}\label{Prop:basedring}
Let $H$ be a cosemisimple Hopf algebra over $\k$ and $\mathcal{S}$ be the set of all the simple subcoalgebras of $H$. Then $\mathbb{Z}\mathcal{S}$ is a unital based ring with $\mathbb{Z}_+$-basis $\mathcal{S}$.
\end{lemma}

Let $\mathcal{F}$ be the free abelian group generated by isomorphism classes of finite-dimensional right $H$-comodules and $\mathcal{F}_0$ the subgroup of $\mathcal{F}$ generated by all expressions $[Y]-[X]-[Z]$, where $0\rightarrow X\rightarrow Y\rightarrow Z\rightarrow0$ is a short exact sequence of finite-dimensional right $H$-comodules.
Recall that the \textit{Grothendieck group} $\operatorname{Gr}(H$-comod$)$ of the category of finite-dimensional right $H$-comodules is defined by $$\operatorname{Gr}(H\text{-comod}):=\mathcal{F}/\mathcal{F}_0.$$

From \cite[Proposition 4.5.4]{EGNO15} and \cite[Theorem 2.7]{Lar71}, $\operatorname{Gr}(H$-comod$)$ is a unital based ring with $\mathbb{Z}_+$-basis $\mathcal{V}$, where $\mathcal{V}$ is the set of all the isomorphism classes of simple right $H$-comodules.
One can show that:
\begin{lemma}\emph{(}\cite[Lemma 2.4]{YL24}\emph{)}\label{lem:ZS}
Let $H$ be a cosemisimple Hopf algebra over $\k$ and $\mathcal{S}$ be the set of all the simple subcoalgebras of $H$. Then \begin{eqnarray*}
F:\operatorname{Gr}(H\text{-comod})&\rightarrow&\mathbb{Z}\mathcal{S},\\
 M\;\;\;\;&\mapsto &\operatorname{cf}(M).
\end{eqnarray*} is an isomorphism between unital based rings.
\end{lemma}

Let $\Lambda$ be the set of all the irreducible characters of simple right $H$-comodules. Define a multiplication on $\mathbb{Z}\Lambda$ as follow: for $\chi(M), \chi(N)\in \Lambda$,
$$\chi(M)\chi(N)=\chi(M\otimes N).$$
Note that for any simple right $H$-comodule $M$, we have $M^*$ is also a simple right $H$-comodule and $$\operatorname{cf}(M^*)=S(\operatorname{cf}(M)).$$ This indicates that $$\chi(M^*)=S(\chi(M)).$$
It is straightforward to show that $\mathbb{Z}\Lambda$ is a unital based ring with $Z_+$-basis $\Lambda$.

As mentioned above, we conclude this subsection by point out the following lemma.
\begin{lemma}\label{lemma:char=ring}
Let $H$ be a cosemisimple Hopf algebra over $\k$ and $\Lambda$ be the set of all the irreducible characters of simple right $H$-comodules. Then \begin{eqnarray*}
F:\operatorname{Gr}(H\text{-comod})&\rightarrow&\mathbb{Z}\Lambda,\\
 M\;\;\;\;&\mapsto &\chi(M).
\end{eqnarray*} is an isomorphism between unital based rings.
\end{lemma}

\subsection{Matched pair of groups}\label{subsection2.3}
\begin{definition}\emph{(}\cite[Definition 2.1]{Tak81}\emph{)}
A matched pair of groups, i.e. a quadruple $(F, G,\triangleleft, \triangleright )$, is a pair of groups $(F, G)$ together with two group actions
\begin{eqnarray*}
&&G\times F\rightarrow F,\;\;(g, f)\mapsto g\triangleright f,\\
&&G\times F\rightarrow G,\;\;(g, f)\mapsto g\triangleleft f
\end{eqnarray*}
on the sets, satisfying the following conditions
 \begin{eqnarray*}
 g\triangleright(ff^\prime)=(g\triangleright f)((g\triangleleft f)\triangleright f^\prime),\;\;(gg^\prime)\triangleleft f=(g\triangleleft (g^\prime \triangleright f))(g^\prime\triangleleft f),
\end{eqnarray*}
for any $g, g^\prime\in G$, $f, f^\prime\in F$.
\end{definition}
The following lemma is proved from the definitions using direct computation.
\begin{lemma}\label{lem:actioninverse}
Let $(F, G,\triangleleft, \triangleright )$ be a matched pair of groups. Then for any $f\in F, g\in G$, the following properties hold:
\begin{itemize}
  \item [(1)]$1_G \triangleright f =f$, $1_G \triangleleft f= 1_G$;
  \item [(2)]$g\triangleright 1_F=1_F$, $g\triangleleft 1_F=g$;
  \item [(3)]$(g\triangleright f)^{-1}=(g\triangleleft f)\triangleright f^{-1}$;
  \item [(4)]$(g\triangleleft f)^{-1}=g^{-1}\triangleleft (g\triangleright f)$.
\end{itemize}
\end{lemma}

Inspired by \cite{JM09, Orl24}, we use the following definitions.
\begin{definition}
Let $(F, G,\triangleleft, \triangleright )$ be a matched pair of groups. Then for any $g\in G, f\in F$, we define the followings:
\begin{itemize}
  \item [(1)]The orbits $O_{f}=\{g\triangleright f\mid g\in G\}$;
  \item [(2)]The stabilizers $G_{f}=\{g\in G\mid g\triangleright f=f\}$;
  \item [(3)]The sets $G_{f, f^{-1}}=\{g\in G\mid g\triangleright f=f^{-1}\}$;
  \item [(4)]A complete set $T_f$ of right coset representatives of $G_f$ in $G$. For convenience, we assume $1_G\in T_f.$
\end{itemize}
\end{definition}

Let $f\in F$ and consider $O_f$. Note that if $x\in O_f$, then $$x=g\triangleright f$$ for some $g\in G.$ It follows from lemma \ref{lem:actioninverse} that $$x^{-1}=(g\triangleright f)^{-1}=(g\triangleleft f)\triangleright f^{-1}.$$ This means that $x^{-1}\in O_{f^{-1}}.$
Using the same argument as in the proof of \cite[Lemma 4.2 and Corollary 4.3]{JM09}, we can easily carry out the proof of the following lemma.

\begin{lemma}\label{lem:Gff-1}
Let $(F, G,\triangleleft, \triangleright )$ be a matched pair of groups. Let $f\in F$ and consider $O_f$. Then the followings are equivalent:
\begin{itemize}
  \item [(1)]$f^{-1}\in O_f$, or $\mid G_{f, f^{-1}}\mid\neq 0 $;
  \item [(2)]$x^{-1}\in O_f$ for some $x\in O_f$;
  \item [(3)]$x^{-1}\in O_f$ for all $x\in O_f$.
\end{itemize}
\end{lemma}
We end this subsection by giving the following lemma, which describes a property of multiplication of two orbits.
\begin{lemma}\label{lem:multiunion}
Let $(F, G,\triangleleft, \triangleright )$ be a matched pair of groups. Suppose $f, f^\prime\in F$ and $O_fO_{f^\prime}=\{xy\mid x\in O_f, y\in O_{f^\prime}\}$. Then there exists a subset $O_{f, f^\prime}$ of $F$ such that  $O_fO_{f^\prime}=\cup_{f^{\prime\prime}\in O_{f, f^\prime}} O_{f^{\prime\prime}}$ and the union is disjoint.
\end{lemma}
\begin{proof}
For any $g\in G$, $x\in O_f$ and $y\in O_{f^\prime}$, we have
\begin{eqnarray*}
g\triangleright (xy)=(g\triangleright x)((g\triangleleft x)\triangleright y)\in O_f O_{f^\prime}.
\end{eqnarray*}
This means that $$O_{f^{\prime\prime}}\subseteq O_fO_{f^\prime}$$ for any $f^{\prime\prime}\in O_fO_{f^\prime}$. Thus there exists a subset $O_{f, f^\prime}$ of $F$ such that  $$O_fO_{f^\prime}=\cup_{f^{\prime\prime}\in O_{f, f^\prime}} O_{f^{\prime\prime}}$$ and the union is disjoint.
\end{proof}

\section{The Hopf algebra $\Bbbk^G{}^\tau\#_{\sigma}\Bbbk F$}\label{section3}
\subsection{Abelian extensions}
In order to describe the extensions in which we are interested, we first consider arbitrary extensions of Hopf algebras.
\begin{definition}\emph{(}\cite{AD95}\emph{)}
A short exact sequence of Hopf algebras is a sequence of Hopf algebras
and Hopf algebra maps
\begin{eqnarray}\label{ext}
\;\; K\xrightarrow{\iota} H \xrightarrow{\pi} A
\end{eqnarray}
such that
\begin{itemize}
  \item[(i)] $\iota$ is injective,
  \item[(ii)]  $\pi$ is surjective,
  \item[(iii)] $\ker(\pi)= HK^+$, $K^+$ is the kernel of the counit of $K$,
  \item[(iv)] $K=\{x\in H \mid (\pi\otimes \id)\Delta (x)=1\otimes x\}$.
\end{itemize}
\end{definition}
An extension \eqref{ext} above such that $K$ is commutative and $A$ is cocommutative is called abelian. Let $G, F$ be finite groups and $\k^G$ denote the dual Hopf algebra of $\k G$.
Abelian extensions $$\k^G\xrightarrow{\iota} H\xrightarrow{\pi} \k F$$ of $\k F$ by $\k^G$
were classified by Masuoka (\cite[Proposition 1.5]{Mas02}), and the above $H$ can be expressed as $\k^G\#_{\sigma, \tau}\k F$.

In this section, we will remove the condition that $F$ is finite and attempt to construct a class of infinite-dimensional Hopf algebras through abelian extensions of $\k F$ by $\k^G$.

In the following part, let $G$ be a finite group
and $\{p_g\mid g\in G\}$ be the basis for $\k^G$ dual to the basis of group elements in $\k G$. Note that we have $1_{\k^G}=\sum_{g\in G}p_g$.

Let $\triangleleft :G\times F\rightarrow G$ be an action on the set $G$, which corresponds to an action $\rightharpoonup: \k F\times \k^G\rightarrow \k^G$ so that $$f\rightharpoonup p_g=p_{g\triangleleft f^{-1}},$$ where $f\in F$ and $g\in G.$ We have $\rightharpoonup$ is a weak action, that is,
\begin{eqnarray*}
&&f\rightharpoonup (p_gp_{g^\prime})=(f\rightharpoonup p_g)(f\rightharpoonup p_{g^\prime})=\delta_{g, g^\prime}p_{g\triangleleft f^{-1}},\\
&&f\rightharpoonup (\sum_{g\in G}p_g)=\sum_{g\in G}p_g,\\
&&1_F\rightharpoonup p_g=p_g,
\end{eqnarray*}
for any $f\in F, g, g^\prime\in G$.

Suppose we are given an action $\triangleright:G\times F\rightarrow F$, which corresponds to a coaction $\rho: \k F\rightarrow \k F\otimes \k^G$ such that
$$\rho(f)=\sum_{g\in G}(g\triangleright f)\otimes p_g.$$ In fact, $\rho$ is a weak coaction, that is, the following conditions are fulfilled:
\begin{eqnarray*}
&&(\Delta\otimes \operatorname{id})\rho(f)=m^{24}(\rho\otimes \rho)\Delta(f)=\sum_{g\in G}(g\triangleright f)\otimes (g\triangleright f)\otimes p_g,\\
&&(\varepsilon\otimes\id)\rho (f)=\varepsilon(f)\otimes 1_{\k^G},\\
&&(\id\otimes\varepsilon)\rho(f)=f,
\end{eqnarray*}
for any $f\in F,$ where $m^{24}:\k F\otimes \k^G\otimes \k F\otimes \k^G \rightarrow \k F \otimes \k F\otimes \k^G$ is the linear map $f\otimes p_{g}\otimes f^{\prime}\otimes p_{g^\prime}  \mapsto f\otimes f^\prime\otimes p_{g}p_{g^\prime}.$

Suppose $\sigma:\k F \times \k F \rightarrow \k^G$ is a bilinear map such that
\begin{eqnarray*}
\sigma(1_F, f)=\sigma(f, 1_F)=1_{\k^G},
\end{eqnarray*}
\begin{eqnarray*}
\varepsilon\sigma=\varepsilon\otimes \varepsilon,
\end{eqnarray*}
and
\begin{eqnarray*}
(f\rightharpoonup \sigma(f^\prime, f^{\prime\prime}))\sigma(f, f^\prime f^{\prime\prime})=\sigma(f, f^\prime)\sigma(ff^\prime, f^{\prime\prime}),
\end{eqnarray*}
for any $f, f^\prime, f^{\prime\prime}\in F.$
We can identify $\sigma$ naturally with the map $G\times F\times F\rightarrow \k^{\times}, (g, f, f^\prime)\mapsto \sigma(f, f^\prime)(g),$ and denote the last valve by $\sigma(g; f, f^\prime)$. A direct computation follows that
\begin{eqnarray}\label{eq:cocycle1}
\sigma(g; 1_F, f)=\sigma(g; f, 1_F)=\sigma(1_G; f, f^\prime)=1,
\end{eqnarray}
and
\begin{eqnarray}\label{eq:cocycle2}
\sigma(g\triangleleft f; f^\prime, f^{\prime\prime})\sigma(g; f, f^\prime f^{\prime\prime})=\sigma(g; f, f^\prime)\sigma(g; ff^\prime, f^{\prime\prime}),
\end{eqnarray}
for any $g\in G$ and $f, f^\prime ,f^{\prime \prime}\in F.$
Since $\k^G$ is commutative, it follows that $\k^G$ is a twist $\k F$-module, that is, $$1_F\rightharpoonup p_g=p_g,$$ and $$(f\rightharpoonup (f^\prime\rightharpoonup p_g))\sigma(f, f^\prime)=\sigma(f, f^\prime)(ff^\prime\rightharpoonup p_g),$$ for any $f, f^\prime\in F$ and $g\in G.$

From \cite[Proposition 2.6]{AD95}, the construction of crossed product makes $\k^G\otimes \k F$ into an algebra with unit
$1_{\k^G}\otimes 1_F
$, whose product is given by
\begin{eqnarray}\label{eq:multi}
(p_g\# f)(p_{g^\prime}\# f^\prime)=p_g(f\rightharpoonup p_{g^\prime})\sigma(f, f^\prime)\#ff^\prime
=\delta_{g\triangleleft f, g^\prime}\sigma(g; f, f^\prime) p_{g}\#ff^\prime,
\end{eqnarray}
where $g, g^\prime\in G$ and $f, f^\prime \in F.$

Besides, suppose $\tau:\k F\rightarrow \k^G\otimes \k^G$ is a linear map satisfying
\begin{eqnarray*}
(\varepsilon\otimes \id)\tau(f)=(\id\otimes \varepsilon)\tau(f)=1_{\k^G},
\end{eqnarray*}
\begin{eqnarray*}
\tau(1_F)=1_{\k^G}\otimes 1_{\k^G},
\end{eqnarray*}
and the co-cocycle condition
\begin{eqnarray*}
\sum_{g\in G} (f_i)_{(1)}(g\triangleright f)_j\otimes (f_{i})_{(2)}(g\triangleright f)^j\otimes f^ip_g= f_k\otimes (f^k)_{(1)}f_l\otimes (f^k)_{(2)}f^l
\end{eqnarray*}
for any $f\in F$, where $\Delta(h)=h_{(1)}\otimes h_{(2)}$ and $\tau(h)=h_i\otimes h^i$ for any $h\in \k^G.$ It is straightforward to show that $\k F$ is a twisted $\k^G$-comodule, that is,
\begin{eqnarray*}
(\id\otimes m_{{(\k^G)}^{\otimes 2}})(\id\otimes \Delta\otimes \id\otimes \id)(\rho\otimes \tau)\Delta(f)&=&m^{13}_{{(\k^G)}^{\otimes 2}}(\id\otimes \id\otimes\rho\otimes \id)(\tau\otimes \rho)\Delta(f)\\
&=&\sum_{g, g^\prime\in G}(g^{\prime}g\triangleright f)\otimes f_i p_{g^\prime}\otimes f^{i}p_g
\end{eqnarray*}
for any $f\in F$, where $m^{13}_{{(\k^G)}^{\otimes 2}}:\k^G\otimes \k^G\otimes \k F\otimes \k^G\otimes \k^G \rightarrow \k F\otimes \k^G\otimes \k^G$ sends $h\otimes k\otimes f\otimes h^\prime\otimes k^\prime\mapsto f\otimes hh^\prime \otimes kk^\prime.$
We can identify $\tau$ naturally with the map $G\times G\times F\rightarrow \k^{\times}, (g, g^\prime, f)\mapsto \tau(f)(g\otimes g^\prime),$ and denote the last valve by $\tau(g, g^\prime; f)$. We know that
\begin{eqnarray}\label{eq:co-cocycle1}
\tau(1_G, g; f)=\tau(g, 1_G; f)=\tau(g, g^\prime; 1_F)=1,
\end{eqnarray}
\begin{eqnarray}\label{eq:co-cocycle2}
\tau(g, g^\prime; g^{\prime\prime}\triangleright f)\tau(gg^\prime, g^{\prime\prime}; f)=\tau(g, g^\prime g^{\prime\prime}; f)\tau(g^\prime, g^{\prime\prime}; f),
\end{eqnarray}
for any $g, g^\prime, g^{\prime\prime}\in G$ and $f\in F.$

By \cite[Proposition 2.16]{AD95}, the construction of crossed coproduct makes $\k^G\otimes \k F$ into a coalgebra with comultiplication
\begin{eqnarray}\label{eq:comult}
\Delta(p_g\# f)&=&\sum_{x,y\in G}(\tau(f)_i p_{gx^{-1}}\#(y\triangleright f))\otimes (p_x\tau(f)^ip_y\# f)\\
&=&\sum_{x\in G}(\tau(gx^{-1}, x; f) p_{gx^{-1}} \#(x\triangleright f))\otimes ( p_x\# f),
\end{eqnarray}
and counit
\begin{eqnarray}\label{eq:counit}
\varepsilon(p_g\# f)=\varepsilon(p_g)\varepsilon(f)=\delta_{g,1_G},\end{eqnarray}
where $g\in G$ and $f \in F.$

According to \cite[Theorem 2.20]{AD95} (also \cite[Proposition 3.1.12]{And96}), we can now obtain the following proposition immediately.
\begin{proposition}\label{prop:Hopf}
Let $F$ be a group and $G$ be a finite group. Let us also fix a map $\sigma: \k G\times \k F\times \k F\rightarrow \k^{\times}$ and a map $\tau :\k G\times \k G\times \k F\rightarrow \k^{\times}$. Let $\k^G{}^\tau\#_{\sigma}\k F$ denote the vector space $\k^G\otimes \k F$ provided with the unit $ 1_{\k^G}\otimes  1_F$, the multiplication (\ref{eq:multi}), the comultiplication (\ref{eq:comult}) and the counit (\ref{eq:counit}). Then $\k^G{}^\tau\#_{\sigma}\k F$ is a bialgebra if and only if the following conditions hold:
\begin{itemize}
  \item [(i)]$(F, G,\triangleleft, \triangleright )$ is a matched pair of groups;
  \item [(ii)]$\sigma$ satisfies the conditions (\ref{eq:cocycle1}) and (\ref{eq:cocycle2});
   \item[(iii)]$\tau$ satisfies the conditions (\ref{eq:co-cocycle1}) and (\ref{eq:co-cocycle2});
  \item [(iv)]$\sigma$ and $\tau$ satisfy the following compatible condition
  \begin{eqnarray*}
  &&\sigma(gg^\prime; f, f^\prime)\tau(g, g^\prime; ff^\prime)\\
  &=&\sigma(g; g^\prime \triangleright f, (g^\prime \triangleleft f)\triangleright f^\prime)\sigma(g^\prime; f, f^\prime)
  \tau(g, g^\prime; f)\tau(g\triangleleft (g^\prime \triangleright f), g^\prime\triangleleft f; f^\prime),
  \end{eqnarray*}
  for any $g, g^\prime, g^{\prime\prime}\in G$ and $f, f^\prime, f^{\prime\prime}\in F$.
\end{itemize}
In this case, $\k^G{}^\tau\#_{\sigma}\k F$ is a Hopf algebra with the antipode given by
$$
S(p_g\# f)=\sigma(g^{-1}; g\triangleright f, (g\triangleright f)^{-1})^{-1}\tau(g^{-1}, g; f)^{-1}p_{(g\triangleleft f)^{-1}}\# (g\triangleright f)^{-1}
$$
for any $g\in G$ and $f\in F.$
Moreover, let $\iota$ be the linear map $p_g \mapsto p_g\otimes 1_F$ and $\pi$ be the linear map $p_g \otimes f \mapsto \varepsilon(p_g)f$. Then
$$
\k^G \xrightarrow{\iota} \k^G{}^\tau\#_{\sigma}\k F \xrightarrow{\pi} \k F
$$
is an exact sequence of Hopf algebras.
\end{proposition}

\begin{corollary}\label{coro:S2=1}
Let $H=\k^G{}^\tau\#_{\sigma}\k F$ be the Hopf algebra appeared in Proposition \ref{prop:Hopf}. Then $S^2=\id.$
\end{corollary}
\begin{proof}
For any $x\in G, y\in F$, it follows from Lemma \ref{lem:actioninverse} that
\begin{eqnarray*}
(x\triangleleft y)^{-1}\triangleleft (x\triangleright y)^{-1}&=&(x^{-1}\triangleleft (x\triangleright y))\triangleleft (x\triangleright y)^{-1}\\
&=&x^{-1}
\end{eqnarray*}
and
\begin{eqnarray*}
(x\triangleleft y)^{-1}\triangleright(x\triangleright y)^{-1}&=&(x\triangleleft y)^{-1}\triangleright( (x\triangleleft y)\triangleright y^{-1})\\
&=&y^{-1}.
\end{eqnarray*}
Direct computation shows that
\begin{eqnarray*}
S^2(p_x\#y)&=&S(\sigma(x^{-1}; x\triangleright y, (x\triangleright y)^{-1})^{-1}\tau(x^{-1}, x; y)^{-1}p_{(x\triangleleft y)^{-1}}\# (x\triangleright y)^{-1})\\
&=&\sigma(x^{-1}; x\triangleright y, (x\triangleright y)^{-1})^{-1}\tau(x^{-1}, x; y)^{-1}\sigma(x\triangleleft y;y^{-1}, y)^{-1}\\
&&\tau(x\triangleleft y, (x\triangleleft y)^{-1};(x\triangleright y)^{-1})^{-1}p_x\#y.
\end{eqnarray*}
Using (\ref{eq:cocycle2}), one can show that $$\sigma(x\triangleleft y;y^{-1}, y)=\sigma(x; y, y^{-1}).$$
Substitute $x\triangleleft y$ for $g$, $(x\triangleleft y)^{-1}$ for $g^\prime$, $x\triangleleft y$ for $g^{\prime\prime}$ and $y^{-1}$ for $f$ in formular (\ref{eq:co-cocycle2}), we have $$\tau(x\triangleleft y, (x\triangleleft y)^{-1};(x\triangleright y)^{-1})=\tau((x\triangleleft y)^{-1}, x\triangleleft y; y^{-1}). $$
Combining the compatible condition $(iv)$ in Proposition \ref{prop:Hopf} and Lemma \ref{lem:actioninverse}, we know that
\begin{eqnarray*}
1&=&\sigma(x^{-1}; x\triangleright y, (x\triangleleft y)\triangleright y^{-1})\sigma(x;y, y^{-1})\tau(x^{-1}, x;y)
\tau(x^{-1}\triangleleft(x\triangleright y), x\triangleleft y; y^{-1})\\
&=&\sigma(x^{-1}; x\triangleright y, (x\triangleright y)^{-1})\sigma(x;y, y^{-1})\tau(x^{-1}, x;y)\tau((x\triangleleft y)^{-1}, x\triangleleft y; y^{-1}).
 \end{eqnarray*}
It follows that
\begin{eqnarray*}
S^2(p_x\#y)&=&\sigma(x^{-1}; x\triangleright y, (x\triangleright y)^{-1})^{-1}\tau(x^{-1}, x; y)^{-1}\sigma(x\triangleleft y;y^{-1}, y)^{-1}\\
&&\tau(x\triangleleft y, (x\triangleleft y)^{-1};(x\triangleright y)^{-1})^{-1}p_x\#y\\
&=&\sigma(x^{-1}; x\triangleright y, (x\triangleright y)^{-1})^{-1}\tau(x^{-1}, x; y)^{-1}\sigma(x; y, y^{-1})^{-1}\\
&&\tau((x\triangleleft y)^{-1}, x\triangleleft y; y^{-1})^{-1}p_x\#y\\
&=&p_x\#y
\end{eqnarray*}
for any $x\in G, y\in F$ and thus $$S^2=\id.$$
\end{proof}
Let $A, B$ be algebras, and $H$ a Hopf algebra. Recall that $A\subset B$ is a (right) $H$-\textit{extension} if $B$ is a right $H$-comodule algebra with $B^{co H}=A.$ The $H$-extension $A\subset B$ is $H$-\textit{cleft} if there exists a right $H$-comodule map $\chi:H\rightarrow B$ which is (convolution) invertible. The extension $B^{co H}\subset B$ is right $H$-\textit{Galois} if the map $\beta: B\otimes_{B^{co H}} B\rightarrow B\otimes H$, given by $a\otimes b\mapsto (a\otimes 1)\rho(b)$, is bijective. See \cite{Mon93}.

\begin{remark}\rm
According to \cite[Theorem 11]{DT86}, the Hopf algebra $\k^G{}^\tau\#_{\sigma}\k F$ in Proposition \ref{prop:Hopf} is a cleft $\k F$-extension of $\k^G$. Using \cite[Theorem 9]{DT86}, we know that the extension $\k^G\subset\k^G{}^\tau\#_{\sigma}\k F$ is right $\k F$-Galois.
\end{remark}

In the following part, let $\k^G{}^\tau\#_{\sigma}\k F$ be the Hopf algebra defined in Proposition \ref{prop:Hopf} with the described algebra and coalgebra structures.
Next we will show the cosemisimplicity of $\k^G{}^\tau\#_{\sigma}\k F$. Before that, let us recall the concept of left or right integrals.

Let $H$ be a Hopf algebra. Recall that an element $T\in H^*$ is a left or right \textit{integral} on $H$ if $\langle T, h\rangle1_H=h_{(1)}\langle T, h_{(2)}\rangle ,$ or $\langle T, h\rangle1_H=\langle T, h_{(1)}\rangle h_{(2)},$ for any $h\in H.$
It is well-known that $H$ is cosemisimple if and only if there exists a left (or right) integral $T$ on $H$ satisfying $\langle T, 1_H\rangle=1$ (see \cite[Theorem 2.4.6]{Mon93}).

It is direct to see the following proposition.
\begin{proposition}\label{prop:cosemisimple}
The Hopf algebra $\k^G{}^\tau\#_{\sigma}\k F$ defined in Proposition \ref{prop:Hopf} is cosemisimple.
\end{proposition}

\begin{proof}
Let $T: \k^G{}^\tau\#_{\sigma}\k F\rightarrow \k$ be the linear map given by
$$\langle T, p_g\# f\rangle=\frac{1}{\mid G \mid}\langle \sum_{x\in G} x \# p_{1_F}, p_g\# f \rangle
%=\frac{1}{\mid G \mid}\sum_{x\in G}\langle p_g,x\rangle\langle p_{1_F},f\rangle
=\frac{1}{\mid G \mid}\delta_{1_F,f}.$$
Next let us check that $T$ is a left integral on $\k^G{}^\tau\#_{\sigma}\k F$.
In fact, for any $g\in G, f\in F$,
one can show that
\begin{eqnarray*}
\langle T, p_g\# f \rangle (\sum_{x\in G}p_x \#1_F)
&=&\frac{1}{\mid G \mid}\delta_{1_F,f}(\sum_{x\in G}p_x \#1_F)\\
&=&\sum_{x\in G}\tau(gx^{-1}, x; f)p_{g x^{-1}}\#(x\triangleright f)\langle T, p_x\# f\rangle.
\end{eqnarray*}
It follows that $T$ is a left integral on $\k^G{}^\tau\#_{\sigma}\k F$. Moreover, we have
$$
\langle T, \sum_{x\in G}p_x\# 1_F \rangle=1.
$$
Using \cite[Theorem 2.4.6]{Mon93}, we know that $\k^G{}^\tau\#_{\sigma}\k F$ is cosemisimple.
\end{proof}

As mentioned above, let us illustrate it with an example.
\begin{example}\label{example:H(e,f)}\rm
Let $\mathbb{Z}_2=\{g\mid g^2=1\}$. Define group actions $\mathbb{Z}_2\xleftarrow{\triangleleft}\mathbb{Z}_2\times \mathbb{Z} \xrightarrow{\triangleright}\mathbb{Z}$ on the sets by
$$
1\triangleleft i=1,\;\; g\triangleleft i=g,\;\;1\triangleright i=i,\;\;g\triangleright i=-i,
$$
for any $i\in \mathbb{Z}$. Clearly, $(\mathbb{Z}, \mathbb{Z}_2)$ together with group actions $\mathbb{Z}_2\xleftarrow{\triangleleft}\mathbb{Z}_2\times \mathbb{Z} \xrightarrow{\triangleright}\mathbb{Z}$ on the sets is a matched pair. Consider the case when $\sigma$ and $\tau$ are trivial, that is, $$\sigma(i, j)=1$$ and $$\tau(x)=1\otimes 1$$ for any $i, j\in \mathbb{Z}$ and $x\in \mathbb{Z}_2$.
In such a case, let $H(\mathbb{Z}, \mathbb{Z}_2)=\k^{\mathbb{Z}_2}{}^{\tau}\#_{\sigma}\k\mathbb{Z}$ be the Hopf algebra defined in Proposition \ref{prop:Hopf} with the described algebra and coalgebra structures. Combining Propositions \ref{prop:Hopf} and \ref{prop:cosemisimple}, we know that $H(\mathbb{Z}, \mathbb{Z}_2)$ is a cosemisimple Hopf algebra.

For any $i\in \mathbb{Z}$, let $e_i=p_1\otimes (-i)$ and $f_i=p_g\otimes i$. As an algebra, $H(\mathbb{Z}, \mathbb{Z}_2)$ is generated by $e_i, f_i$ for $i\in\mathbb{Z}$, subject to the following relations
 \begin{eqnarray*}
1=e_0+f_0, \;\;e_ie_j=e_{i+j},\;\;f_if_j=f_{i+j},\;\;e_if_j=f_je_i=0,
\end{eqnarray*}
for any $i, j\in\mathbb{Z}$.
The comultiplication, counit and the antipode are given by
\begin{eqnarray*}
\Delta(e_i)=e_i\otimes e_i+f_i\otimes f_{-i},\;\;\varepsilon(e_i)=1,\;\;S(e_i)=e_{-i},
\end{eqnarray*}
\begin{eqnarray*}
\Delta(f_i)=e_i\otimes f_i+f_i\otimes e_{-i},\;\;\varepsilon(f_i)=0,\;\;S(f_i)=f_{i},
\end{eqnarray*}
for any $i\in\mathbb{Z}$.

Denote $x=e_0-f_0$, it is clear that $x$ is a group-like element of order $2$. For any $i\geq 1$, denote
$C_i=\span\{e_i, f_i, e_{-i}, f_{-i}\}$. We can show that each $C_i$ is a simple subcoalgebra with basic multiplicative matrix $\C_i$, where
$$
\C_i=\left(\begin{array}{cc}
e_i&f_i\\
f_{-i}&e_{-i}
 \end{array}\right).
$$
It should be pointed out that $C_i=S(C_i)$ for any $i\geq 1.$
We can obtain $$H(\mathbb{Z}, \mathbb{Z}_2)=\k1\oplus \k x\oplus \bigoplus_{i\geq 1}C_i,$$ which follows that the set $\mathcal{S}$ of all the simple subcoalgebras is $\{\k1, \k x\}\cup\{ C_i\mid  i\geq 1\}$.
Note that $H(\mathbb{Z}, \mathbb{Z}_2)$ is exactly the coradical of $H(e_{\pm 1}, f_{\pm 1}, u, v)$ in \cite[Definition 5.1]{YL24}.
\end{example}

Now let us recall that the \textit{cotensor product} (see \cite{Tak77}) of a right comodule $M$ and a left comodule $N$ over a coalgebra $H$ is
\begin{eqnarray*}
M \Box_H N= \ker( M\otimes N \xrightarrow{\rho_M\otimes \id-\id\otimes \rho_N} M\otimes H\otimes N),
\end{eqnarray*}
where $\rho_M$ and $\rho_N$ denote the comodule structure maps of $M$ and $N$.
The functors $M \Box_H -$ and $-\Box_H N$ are left exact. Moreover, we have $$M \Box_H H\cong M$$ and $$H\Box_H N\cong N.$$

Recall that with an epimorphism $H\rightarrow H^\prime$ of coalgebras one can associate a functor $-\Box_{H^\prime} H$. We shall say that $H$ is \textit{faithfully coflat} (\cite[Proposition 1.2.11]{AD95}) as an $H^\prime$-comodule, if
\begin{itemize}
  \item [(1)]whenever $M \rightarrow N$ is an epimorphism of $H^\prime$-comodules, then $M \Box_{H^\prime} H\rightarrow N \Box_{H^\prime} H$ is also an epimorphism of $H$-comodules;
  \item [(2)]$M \Box_{H^\prime} H\rightarrow M$ defined by $m\otimes h\mapsto m \varepsilon(h)$ is surjective for any $H^\prime$-comodule $M$.
\end{itemize}

\begin{corollary}
Let $H=\k^G{}^\tau\#_{\sigma}\k F$ be the Hopf algebra appeared in Proposition \ref{prop:Hopf}. Then
\begin{itemize}
  \item [(1)]$H$ is faithfully coflat as $\k F$-comodule and $\k F$ is a conormal quotient Hopf algebra of $H$.
  \item [(2)] $\k^G$ is a normal Hopf subalgebra of $H$ and $H$ is faithfully flat as $\k^G$-module.
\end{itemize}
\end{corollary}
\begin{proof}
It is straightforward to show that $\k^G$ is a normal Hopf subalgebra of $H$ and $\k F$ is a conormal quotient Hopf algebra of $H$.
It follows from \cite[Theorem 2.1]{Chi14} that $H$ is faithfully flat as $\k^G$-module. Using \cite[Theorem 1.4]{Sch93}, we know that $H$ is faithfully coflat as $\k F$-comodule and $\k F$ is a conormal quotient Hopf algebra of $H$ is equivalent to $\k^G$ is a normal Hopf subalgebra of $H$ and $H$ is faithfully flat as $\k^G$-module. Thus the proof is completed.
\end{proof}

\subsection{Compact quantum groups}
In this subsection, let $\k$ be the field of complex numbers $\mathbb{C}$.

We recall that a $*$-\textit{Hopf algebra} is a pair $(H, *)$, where $H$ is a Hopf algebra over $\mathbb{C}$ and $*:H\rightarrow H$ is a conjugate linear involution which is antimultiplicative and comultiplicative. A $*$-\textit{Hopf algebra morphism} is a Hopf algebra morphism that is also a $*$-map.

Suppose $H$ is a $*$-Hopf algebra, and $V$ is a right $H$-comodule endowed with an inner product $\langle-,-\rangle$. $V$ is called \textit{unitarizable} if
$$\sum\langle v_{(0)}, w\rangle S(v_{1})=\sum \langle v, w_{(0)}\rangle w_{1}^*$$
for any $v, w\in V.$

A \textit{compact quantum group} is a $*$-Hopf algebra $(H, *)$ such that every finite-dimensional right $H$-comodule of $H$ is unitarizable (see \cite[\textsection 11 Theorem 27]{KS97}, also \cite{DK94}).
Let $T$ be a left integral on $H$. By \cite[Section 2]{And94}, the sesquilinear form $\langle ,\rangle_r$ defined by
\begin{eqnarray*}
   \langle x, y\rangle_r =\langle T, y^*x\rangle
\end{eqnarray*}
 is Hermitian.
From \cite[Proposition 2.4]{And94}, $H$ is a compact quantum group if and only if $\langle x, y\rangle_r$ is positive defined.

We shall say that a sequence of morphisms of compact quantum groups ($*$-Hopf algebras)
$$\;\; K\xrightarrow{\iota} H \xrightarrow{\pi} A$$
is \textit{exact} if it is as a sequence of the underlying Hopf algebras.
\begin{example}\emph{(}\cite[Example 2.9]{AFH12}\emph{)}\rm
\begin{itemize}
  \item[(1)] If $G$ is a finite group, then the Hopf algebra $\mathbb{C}^G$ is a compact quantum group with the structure $(p_g)^*=p_g$, for all $g\in G.$
  \item[(2)] Let $F$ be a group. The group algebra $\mathbb{C} F$ is a compact quantum group with respect to $f^*=f^{-1},$ for all $f\in F.$
\end{itemize}
\end{example}
The authors in \cite[Theorem 3]{Kac68} (also \cite[Theorem 6.14, Observation 6.15]{AFH12}) showed that $\mathbb{C}^G{}^\tau\#_{\sigma}\mathbb{C} F$ is a compact quantum group under the assumption that $G, F$ are finite groups. In fact, their results still hold even if $F$ is not finite.
\begin{proposition}\label{prop:CQG}
With the notations in Proposition \ref{prop:Hopf}, let $\mathbb{C}^G{}^\tau\#_{\sigma}\mathbb{C} F$ be the Hopf algebra with the described algebra and coalgebra structures. Then the formula
\begin{eqnarray*}
(p_g\# f)^*=\overline{\sigma(g; f, f^{-1})}p_{g\triangleleft f}\# f^{-1}
\end{eqnarray*}
defines a compact quantum group structure on $\mathbb{C}^G{}^\tau\#_{\sigma}\mathbb{C} F$ if and only if
\begin{eqnarray}\label{eq:tau=sigma=1}
\mid\sigma(g; f, f^\prime)\mid=\mid\tau(g, g^\prime; f)\mid=1,
\end{eqnarray}
for all $g, g^\prime\in G, f, f^\prime \in F$. In this situation, $$\mathbb{C}^G\rightarrow \mathbb{C}^G{}^\tau\#_{\sigma}\mathbb{C} F\rightarrow \mathbb{C} F$$ is a short exact sequence of compact quantum groups.
\end{proposition}
\begin{proof}
Suppose \begin{eqnarray*}
\mid\sigma(g; f, f^\prime)\mid=\mid\tau(g, g^\prime; f)\mid=1,
\end{eqnarray*}
for all $g, g^\prime\in G, f, f^\prime \in F$.
It follows from \cite[Proposition 3.2.9]{And96} that $\mathbb{C}^G{}^\tau\#_{\sigma}\mathbb{C} F$ is a $*$-Hopf algebra and $$\mathbb{C}^G\rightarrow \mathbb{C}^G{}^\tau\#_{\sigma}\mathbb{C} F\rightarrow \mathbb{C} F$$ is a short exact sequence of $*$-Hopf algebras. Conversely, if $\mathbb{C}^G{}^\tau\#_{\sigma}\mathbb{C} F$ is a $*$-Hopf algebra, (\ref{eq:tau=sigma=1}) is got from the proof of \cite[Theorems 5.7, 6.14, Observation 6.15]{AFH12}.
Besides, note that $T$ defined in the proof of Proposition \ref{prop:cosemisimple} is a left integral on $\mathbb{C}^G{}^\tau\#_{\sigma}\mathbb{C} F$.
Note that for any $g, g^\prime\in G, f, f^\prime \in F$, we have
\begin{eqnarray*}
&&\langle T, (p_g \# f)^* (p_{g^\prime}\# f^\prime) \rangle\\
&=&\langle T, \overline{\sigma(g; f, f^{-1})}(p_{g\triangleleft f}\# f^{-1})(p_{g^\prime}\# f^\prime)\rangle\\
&=&\langle T, \overline{\sigma(g; f, f^{-1})}\delta_{g, g^\prime}\sigma(g\triangleleft f;f^{-1}, f^\prime)p_{g\triangleleft f} \#f^{-1}f^\prime \rangle.
\end{eqnarray*}
It follows that  $$\langle T, (p_g \# f)^* (p_{g^\prime}\# f^\prime) \rangle\neq 0$$ if and only if
$g=g^\prime$ and $f=f^\prime.$
For any $\sum\limits_{g\in G}\sum\limits_{f\in F}a_{g, f}p_g\#f\in \mathbb{C}^G{}^\tau\#_{\sigma}\mathbb{C} F,$ we have
\begin{eqnarray*}
&&\langle T, (\sum\limits_{g\in G}\sum\limits_{f\in F}a_{g, f}p_g\#f)^* (\sum\limits_{g\in G}\sum\limits_{f\in F}a_{g, f}p_g\#f) \rangle\\
 &=& \sum\limits_{g\in G}\sum\limits_{f\in F} \mid a_{g, f}\mid^2\langle T, (p_g\#f)^*(p_g\#f)\rangle\\
 &=&\sum\limits_{g\in G}\sum\limits_{f\in F} \mid a_{g, f}\mid^2 \langle T, \overline{\sigma(g; f, f^{-1})}\sigma(g\triangleleft f , f^{-1}, f) p_{g\triangleleft f}\#1_F\rangle\\
&=&\sum\limits_{g\in G}\sum\limits_{f\in F} \mid a_{g, f}\mid^2 \frac{1}{\mid G\mid}> 0.
\end{eqnarray*}
According to \cite[Proposition 2.4]{And94}, we know that $\mathbb{C}^G{}^\tau\#_{\sigma}\mathbb{C}F$ is a compact quantum group.
\end{proof}
\begin{remark}\rm
Recall that the Haar state $\alpha$ (see \cite[\textsection 11 Proposition 29]{KS97}) of a compact quantum group $H$ is a linear functional on $H$ satisfying $\alpha(1_H)=1$ and $\alpha(h^*h)>0$ for all nonzero $h\in H.$
The left integral $T$ defined in the proof of Proposition \ref{prop:cosemisimple} is exactly the Haar state of the compact quantum group $\mathbb{C}^G{}^\tau\#_{\sigma}\mathbb{C} F$.
\end{remark}

\section{Comodules over $\Bbbk^G{}^\tau\#_{\sigma}\Bbbk F$}\label{section4}
With the notations in Proposition \ref{prop:Hopf}, let $H=\k^G{}^\tau\#_{\sigma}\k F$ be the Hopf algebra with the described algebra and coalgebra structures. We use the notations in Subsection \ref{subsection2.3} for convenience. In this section, we will give a description for the simple right $\k^G{}^\tau\#_{\sigma}\k F$-comodules.

Before proceeding further, let us give the following lemma.
\begin{lemma}\label{lem:Cf}
With the notations in Proposition \ref{prop:Hopf}, let $\k^G{}^\tau\#_{\sigma}\k F$ be the Hopf algebra with the described algebra and coalgebra structures.
For any $f\in F$, let $C_f=\span\{p_g\# (g^\prime \triangleright f)\mid g, g^\prime\in G\}$.
\begin{itemize}
  \item [(1)]For any $f\in F$, $C_f$ is a cosemisimple subcoalgebra of $\k^G{}^\tau\#_{\sigma}\k F$. In particular,
  $C_{1_F}$ is a cosemisimple Hopf subalgebra of $H$ and $C_{1_F}\cong \k^G$ as Hopf algebras.
  \item [(2)]If the stabilizer $G_{f}=\{1_G\}$, then $C_f$ is a simple subcoalgebra. In this case, we have $S(C_f)=C_f$ if and only if $G_{f, f^{-1}}\neq \emptyset.$
\end{itemize}
\end{lemma}
\begin{proof}
\begin{itemize}
  \item [(1)]It is a consequence of Proposition \ref{prop:cosemisimple} that $C_f$ is a cosemisimple subcoalgebra of $\k^G{}^\tau\#_{\sigma}\k F$.
  By definition, we know that
  $$
\Delta(p_g\# 1_F)=\sum_{x\in G} (p_{gx^{-1}} \#1_F)\otimes ( p_x\# 1_F)\in C_{1_F}\otimes C_{1_F},
$$
$$(p_g\# 1_F)(p_{g^\prime}\# 1_F)
=\delta_{g, g^\prime} p_{g}\#1_F\in C_{1_F},$$
and
$$
S(p_g\# 1_F)=p_{g^{-1}}\# 1_F\in C_{1_F},
$$
where $g, g^\prime\in G$. It follows that $C_{1_F}$ is a Hopf subalgebra of $\k^G{}^\tau\#_{\sigma}\k F$. Moreover, define
\begin{eqnarray*}
\alpha:\;\;\; C_{1_F}\;\; &\rightarrow& \k^{G}\\
p_g\#1_F&\mapsto& p_g.
\end{eqnarray*}
It is straightforward to show that $\alpha$ is a Hopf algebra isomorphism.
  \item [(2)]If $G_{f}=\{1_G\},$ then the set $\{p_g\# (g^\prime \triangleright f)\mid g, g^\prime\in G\}$ is linearly independent, which indicates that $C_f$ is a simple subcoalgebra. Suppose $$S(C_f)=C_f,$$ we know that $$S(p_{1_G}\#(1_G\triangleright f))=p_{1_G}\# f^{-1}\in C_f.$$ It follows that $$f^{-1}\in O_f,$$ and thus $$G_{f, f^{-1}}\neq \emptyset.$$
      Conversely, suppose there exists some $g\in G_{f, f^{-1}}.$ According to Lemma \ref{lem:actioninverse}, we have
      \begin{eqnarray*}
      (g^\prime g^{\prime\prime}\triangleright f)^{-1}=((g^\prime g^{\prime\prime})\triangleleft f)\triangleright f^{-1}=((g^\prime g^{\prime\prime})\triangleleft f)\triangleright(g\triangleright f)\in O_f,
      \end{eqnarray*}
      for any $g^\prime, g^{\prime\prime}\in G.$
      As a result, we know that
      \begin{eqnarray*}
     \;\;\;\;\;\;\;\;\;&& S(p_{g^\prime}\#( g^{\prime\prime}\triangleright f))\\
     &=&\sigma((g^\prime)^{-1}; g^\prime g^{\prime\prime}\triangleright f, (g^\prime g^{\prime\prime}\triangleright f)^{-1})^{-1}\tau((g^\prime)^{-1}, g^\prime; g^{\prime\prime}\triangleright f)^{-1}p_{(g^{\prime}\triangleleft (g^{\prime\prime}\triangleright f))^{-1}}\#(g^\prime g^{\prime\prime}\triangleright f)^{-1}\\
     &\in& C_f,
      \end{eqnarray*}
       for any $g^\prime, g^{\prime\prime}\in G.$ Using the fact that $S$ is bijective, one can show that $$S(C_f)=C_f.$$
 \end{itemize}
\end{proof}

For any $f\in F$, let $\k^{G_f}{}^\tau\#\k F$ be the linear space spanned by $\{p_g\# f^\prime\mid g\in G_f, f^\prime\in F\}$. The construction of crossed coproduct makes $\k^{G_f}{}^\tau\#\k F$ into a coalgebra with comultiplication
$$
\Delta(p_g\# f^\prime)=\sum_{x\in G_f}(\tau(gx^{-1}, x; f^\prime) p_{gx^{-1}} \#(x\triangleright f^\prime))\otimes ( p_x\# f^\prime),
$$
and counit
$$
\varepsilon(p_g\# f^\prime)=\delta_{g,1_G},
$$
where $g\in G_f, f^\prime\in F$.

Define
\begin{eqnarray*}
\pi:\;\;\;\; \k^{G}{}^\tau\#_{\sigma}\k F &\rightarrow& \k^{G_f}{}^\tau\#\k F\\
\sum_{g\in G}\sum_{f^\prime\in F}a_{g, f^\prime}p_g\#f^\prime&\mapsto&\sum_{g\in G_f}\sum_{f^\prime\in F}a_{g, f^\prime}p_g\#f^\prime.
\end{eqnarray*}
We can show that $\pi$ is an epimorphism of coalgebras. It follows that $\k^{G}{}^\tau\#_{\sigma}\k F$ is a left $\k^{G_f}{}^\tau\#\k F$-comodule, whose comodule structure is defined by
  \begin{eqnarray*}
 \rho(p_g\#f^\prime)=(\pi\otimes \id)\Delta(p_g\#f^\prime),
  \end{eqnarray*}
  for any $g\in G, f^\prime\in F.$

In \cite[Lemma 3.2, Theorem 3.3]{KMM02}, the authors characterized all the simple left $\k^G{}^\tau\#_{\sigma}\k F$-modules when $G, F$ are finite groups. Dually, we attempt to describe all the simple right  $\k^G{}^\tau\#_{\sigma}\k F$-comodules even if $F$ is not finite. We will show that all simple right $\k^G{}^\tau\#_{\sigma}\k F$-comodules may be described as induced comodules from $\k_{\tau_f}^{G_f}$, where $\k_{\tau_f}^{G_f}$ is the coalgebra over $\k^{G_f}$ with the revised comultiplication
$$\Delta(p_g)=\sum_{x\in G_f}\tau(gx^{-1}, x; f) p_{gx^{-1}}\otimes p_{x}.$$

It will be shown in the following lemma that any right $\k_{\tau_f}^{G_f}$-comodule can induce a right $\k^G{}^\tau\#_{\sigma}\k F$-comodule.
\begin{lemma}\label{lem:inducedcomodule}
Let $H=\k^G{}^\tau\#_{\sigma}\k F$ be the Hopf algebra as in Proposition \ref{prop:Hopf}, and fix an element $f\in F.$ Let $H^\prime=\k^{G_f}{}^\tau\#\k F$ and $\k_{\tau_f}^{G_f}$ be the coalgebras defined above. Then
\begin{itemize}
  \item [(1)]$\k_{\tau_f}^{G_f}$ is isomorphic as a coalgebra to $\k^{G_f}{}^\tau\#\k f$, where $\k^{G_f}{}^\tau\#\k f$ is the subcoalgebra of $H^\prime$ spanned by $\{p_g\#f\mid g\in G_f\}$.
  \item [(2)]Let $(V, \rho)$ be a right $\k_{\tau_f}^{G_f}$-comodule, and let $V^\prime=V\otimes \k f$. Then $V^\prime$ becomes a right $H^\prime$-comodule by defining, for each $v\in V$,
      $$
      \rho^\prime(v\otimes f)=\sum_{g\in G_{f}} v_g\otimes f\otimes p_g\# f,
      $$
      where $\rho(v)=\sum_{g\in G_{f}} v_g\otimes p_g.$
  \item [(3)]Let $\tilde{V}:=V^\prime \Box_{H^\prime} H=(V\otimes \k f)\Box_{H^\prime} H$. Then we may write $\tilde{V}=\sum_{z\in T_f} V\otimes \k f\otimes z.$ Moreover, $\tilde{V}$ becomes a right $H$-comodule whose comodule structure is induced by the comultiplication of $H$.
      For any $x\in G$, suppose $x=g_xz_x,$ where $g_x\in G_f$ and $z_x\in T_f$. Then the comodule structure on $\tilde{V}$ is given by the following formula:
      \begin{eqnarray*}
      &&\tilde{\rho}(v\otimes f\otimes z)\\
      &=&\sum_{x\in G} \tau(z_x^{-1},g_x^{-1};f)^{-1}\tau(z_x^{-1}g_x^{-1}z, z^{-1};f)v_{g_x^{-1}}\otimes f\otimes z_x \otimes p_{x^{-1}z}\#(z^{-1}\triangleright f),
      \end{eqnarray*}
      for any $v\in V$ and $z\in T_f$.
\end{itemize}
\end{lemma}

\begin{proof}
\begin{itemize}
  \item [(1)]Define
\begin{eqnarray*}
\gamma:\k_{\tau_f}^{G_f}&\rightarrow& \k^{G_f}{}^\tau\#\k f\\
 p_g\;&\mapsto &\;p_g\# f.
\end{eqnarray*}
It is straightforward to show that $f$ is a coalgebra isomorphism.
  \item [(2)]Since $(V, \rho)$ is a right $\k_{\tau_f}^{G_f}$-comodule, it follows that
  $$\sum_{g, x\in G_f}\tau(gx^{-1}, x; f)v_g\otimes p_{gx^{-1}}\otimes p_{x}=\sum_{g, y\in G_f}(v_g)_{y}\otimes p_{y}\otimes p_g,$$
  and
  $$
  v=\sum_{g\in G_f}\varepsilon(p_g) v_g=v_{1_G},
  $$
  for any $v\in V.$
  Clearly we have $$(v_{g})_y=\tau(y, g; f)v_{yg}$$ for any $g, y\in G_f.$
  From this, we deduce
  $$
  \sum_{g, x\in G_f}\tau(gx^{-1}, x; f)v_{g}\otimes f\otimes p_{gx^{-1}}\#f\otimes p_{x}\#f=\sum_{g, y\in G_f} (v_{g})_y\otimes f\otimes p_y\# f \otimes p_g\# f,
  $$
  and
  $$
  v\otimes f=\sum_{g\in G_f}\varepsilon(p_g\# f)v_g\otimes f=v_{1_G}\otimes f,
  $$
  for any $v\in V.$ It turns out that $V^\prime$ is a right $H^\prime$-comodule.
  \item [(3)]
 For any $z\in T_f$, let $(\k^{G_f z}){}^\tau\#\k F$ be the subspace of $(\k^{G}){}^\tau\#\k F$ spanned by $\{p_g\# f^\prime\mid g\in G_f z, f^\prime\in F\}$. Define
  $$
  \delta(p_g\# f^\prime)=(\pi \otimes \id)\Delta(p_g\# f^\prime)
  $$
  for any $g\in G_f z, f^\prime\in F.$
 We know that $(\k^{G_f z}){}^\tau\#\k F$ is a left $H^\prime$-comodule and
  $$
  H=\k^{G}{}^\tau\#_\sigma\k F= \bigoplus_{z\in T_f}(\k^{G_f z}){}^\tau\#\k F
  $$
  as left $H^\prime$-comodule.
  For any $z\in T_f$, define
  \begin{eqnarray*}
  \gamma_z:(\k^{G_f z}){}^\tau\#\k F&\rightarrow& \k^{G_f}{}^\tau\#\k F\\
  \sum_{g\in G_f,f^\prime\in F}a_{g, f^\prime}p_{gz}\#(z^{-1}\triangleright f^\prime)&\mapsto&\sum_{g\in G_f,f^\prime\in F}\tau(gz, z^{-1};f^\prime)^{-1}a_{g,f^\prime}p_{g}\#f^\prime.
  \end{eqnarray*}
  It is clear that $\gamma_z$ is a linear isomorphism.
  We claim that $\gamma_z$ is a left $H$-comodule isomorphism. In fact,
it follows from (\ref{eq:co-cocycle2}) that
$$\tau(gx^{-1},xz;z^{-1}\triangleright f)\tau(gz,z^{-1};f)=\tau(gx^{-1},x;f)\tau(xz,z^{-1};f).$$
For any $g\in G_f, f^\prime\in F, z\in T_f$, we have
 \begin{eqnarray*}
&&(\id\otimes \gamma_z)\delta(p_{gz}\#(z^{-1}\triangleright f^\prime))\\
&=&\sum_{x\in G_f}\tau(gx^{-1}, xz; z^{-1}\triangleright f^\prime)\tau(xz, z^{-1};f^\prime)^{-1}p_{gx^{-1}}\#(x\triangleright f^\prime)\otimes p_{x}\# f^\prime\\
&=&\sum_{x\in G_f}\tau(gx^{-1}, x;f^\prime)\tau(gz, z^{-1};f^\prime)^{-1}p_{gx^{-1}}\#(x\triangleright f^\prime)\otimes p_{x}\# f^\prime\\
&=&\delta\gamma_z(p_{gz}\#(z^{-1}\triangleright f^\prime)).
\end{eqnarray*}
Thus $\gamma_z$ is a left $H$-comodule isomorphism and we have
$$
H= \bigoplus_{z\in T_f}(\k^{G_f z}){}^\tau\#\k F \cong (H^\prime)^{\oplus \mid T_f \mid}.
$$
It follows that
$$\tilde{V}=(V\otimes \k f)\Box_{H^\prime} H \cong (V\otimes \k f)^{\oplus \mid T_f \mid}$$
as right $H$-comodules. This indicates that we may write $$\tilde{V}=\sum_{z\in T_f} V\otimes \k f\otimes z.$$
Moreover, for any $v\in V$ and $z\in T_f$, we know that
\begin{eqnarray*}
(\id\otimes \gamma_{z}^{-1})\rho^\prime(v\otimes f)&=&\sum_{g\in G_f}\tau(gz, z^{-1};f)v_g\otimes f\otimes p_{gz}\#(z^{-1}\triangleright f) \\
&\in& (V\otimes \k f)\Box_{H^\prime} ((\k^{G_f z}){}^\tau\#\k F).
\end{eqnarray*}
Note that $$\dim_{\k}((V\otimes \k f)\Box_{H^\prime} ((\k^{G_f z}){}^\tau\#\k F))=\dim_{\k}(V).$$
Then for any $z\in T_f$, the subspace
$(V\otimes \k f)\Box_{H^\prime} ((\k^{G_f z}){}^\tau\#\k F)$ of $(V\otimes \k f)\Box_{H^\prime} H$ is spanned by $\{\sum_{g\in G_f}\tau(gz, z^{-1};f)v_g\otimes f\otimes p_{gz}\#(z^{-1}\triangleright f)\mid v\in V\}.
$
It follows that
\begin{eqnarray*}
&&(V\otimes \k f)\Box_{H^\prime} H=\span\{\sum_{g\in G_f}\tau(gz, z^{-1};f)v_g\otimes f\otimes p_{gz}\#(z^{-1}\triangleright f)\mid v\in V, z\in T_f\}.
\end{eqnarray*}
In particular, for any $g\in G_f$, we have
\begin{eqnarray*}
&&(\id\otimes \gamma_{z}^{-1})\rho^\prime(v_g\otimes f)\\
&=&\sum_{x\in G_f} \tau(xz, z^{-1};f) (v_g)_x\otimes f \otimes p_{xz}\otimes (z^{-1}\triangleright f)\\
&=&\sum_{x\in G_f}\tau(x, g; f) \tau(xz, z^{-1};f) v_{xg}\otimes f \otimes p_{xz}\otimes (z^{-1}\triangleright f)\\
&=&\sum_{x\in G_f}\tau(xg^{-1}, g; f) \tau(xg^{-1}z, z^{-1};f) v_{x}\otimes f \otimes p_{xg^{-1}z}\otimes (z^{-1}\triangleright f)\\
&\in& (V\otimes \k f)\Box_{H^\prime} ((\k^{G_f z}){}^\tau\#\k F).
\end{eqnarray*}
For any $\sum_{g\in G_f}\tau(gz, z^{-1};f) v_g\otimes f\otimes p_{g z}\#(z^{-1}\triangleright f )\in (V\otimes \k f)\Box_{H^\prime} ((\k^{G_f z}){}^\tau\#\k F)$, define
\begin{eqnarray*}
&&\tilde{\rho}(\sum_{g\in G_f}\tau(gz, z^{-1};f)v_g\otimes f\otimes p_{g z}\#(z^{-1}\triangleright f))\\
&=&\sum_{g\in G_f}\tau(gz, z^{-1};f)v_g\otimes f\otimes \Delta(p_{gz}\#(z^{-1}\triangleright f))\\
&=&\sum_{g\in G_f}\sum_{x\in G}\tau(gz, z^{-1};f)\tau(gx, x^{-1}z; z^{-1}\triangleright f)v_g\otimes f\otimes p_{gx}\#(x^{-1}\triangleright f)\\
&&\otimes p_{x^{-1}z}\#(z^{-1}\triangleright f)\\
&=&\sum_{g\in G_f}\sum_{x\in G}\tau(gz, z^{-1};f)\tau(gx, x^{-1}z; z^{-1}\triangleright f)v_g\otimes f\otimes p_{gg_xz_x}\#(z_x^{-1}\triangleright f)\\
&&\otimes p_{x^{-1}z}\#(z^{-1}\triangleright f).\\
\end{eqnarray*}
Performing a direct substitution in (\ref{eq:co-cocycle2}), we obtain
\begin{eqnarray*}
&&\tau(gz, z^{-1};f)\tau(gx, x^{-1}z; z^{-1}\triangleright f)\\&=&\tau(gx, x^{-1}; f)\tau(x^{-1}z, z^{-1};f)\\
&=&\tau(gg_xz_x, z_x^{-1}g_x^{-1}; f)\tau(z_x^{-1}g_x^{-1}z, z^{-1};f)\\
&=&\tau(gg_x, g_x^{-1};f)\tau(gg_xz_x, z_{x}^{-1};f)\tau(z_x^{-1},g_x^{-1};f)^{-1}\tau(z_x^{-1}g_x^{-1}z, z^{-1};f),
\end{eqnarray*}
for any $g\in G_f, z\in T_f$, $x\in G$ and $f\in F$.
It follows that
\begin{eqnarray*}
\;\;\;\;\;\;\;\;&&\tilde{\rho}(\sum_{g\in G_f}\tau(gz, z^{-1};f)v_g\otimes f\otimes p_{g z}\#(z^{-1}\triangleright f))\\
&=&\sum_{x\in G}\tau(z_x^{-1},g_x^{-1};f)^{-1}\tau(z_x^{-1}g_x^{-1}z, z^{-1};f)\\
 &&(\sum_{g\in G_f} \tau(gg_x, g_x^{-1};f)\tau(gg_xz_x, z_{x}^{-1};f) v_g\otimes f\otimes p_{gg_xz_x}\#(z_x^{-1}\triangleright f))
\otimes p_{x^{-1}z}\#(z^{-1}\triangleright f).\\
&\in& ((V\otimes \k f)\Box_{H^\prime} H)\otimes H.
\end{eqnarray*}
The fact that $\Delta$ is coassociative implies that $\tilde{V}$ is a right $H$-comodule. Besides, the comodule structure on $\tilde{V}$ is given by the following formula:
      \begin{eqnarray*}
     && \tilde{\rho}(v\otimes f\otimes z)\\
     &=&\sum_{x\in G} \tau(z_x^{-1},g_x^{-1};f)^{-1}\tau(z_x^{-1}g_x^{-1}z, z^{-1};f)v_{g_x^{-1}}\otimes f\otimes z_x \otimes p_{x^{-1}z}\#(z^{-1}\triangleright f),
      \end{eqnarray*}
      for any $v\in V$ and $z\in T_f$.
\end{itemize}
\end{proof}
In the following part, we will fix our notations as in Lemma \ref{lem:inducedcomodule}. We are now in a position to prove the main result in this section.
\begin{theorem}\label{thm:simple}
Let $H=\k^G{}^\tau\#_{\sigma}\k F$ be the Hopf algebra as in Proposition \ref{prop:Hopf}, and fix an element $f\in F.$ For each $G$-orbit $O_f$ of $F$, let $V$ be a right $\k_{\tau_f}^{G_f}$-comodule. Let $\tilde{V}=(V\otimes \k f)\Box_{\k^{G_f}{}^\tau\#\k F} H$ be the induced right $H$-comodule as in Lemma \ref{lem:inducedcomodule}.
\begin{itemize}
  \item [(1)]If $V$ is a simple right $\k_{\tau_f}^{G_f}$-comodule, then $\tilde{V}$ is a simple right $H$-comodule.
  \item [(2)]Every simple right $H$-comodule is isomorphic to $\tilde{V}$ for some simple right $\k_{\tau_f}^{G_f}$-comodule, where $f$ ranges over a choice of one element in each $G$-orbit of $F$.
\end{itemize}
\end{theorem}
\begin{proof}
\begin{itemize}
  \item [(1)]
Suppose that $W$ is a nonzero right $H$-subcomodule of $\tilde{V}$. For any $0\neq w\in W$, according to Lemma \ref{lem:inducedcomodule}, we may write $$w=\sum_{z\in T_f}v^{(z)}\otimes f\otimes z,$$
where $v^{(z)}\in V$ for all $z\in T_f,$ and there exists some $z_{0}\in T_f$ such that $v^{(z_0)}\neq0.$ For any $x\in G$, suppose $x=g_xz_x,$ where $g_x\in G_f$ and $z_x\in T_f$.
It is a consequence of Lemma \ref{lem:inducedcomodule} that
\begin{eqnarray*}
&&\tilde{\rho}(\sum_{z\in T_f}v^{(z)}\otimes f\otimes z)\\
&=&\sum_{z\in T_f}\sum_{x\in G} \tau(z_x^{-1},g_x^{-1};f)^{-1}\tau(z_x^{-1}g_x^{-1}z, z^{-1};f)v^{(z)}_{g_x^{-1}}\otimes f\otimes z_x \otimes p_{x^{-1}z}\#(z^{-1}\triangleright f),
\end{eqnarray*}
where $\rho(v^{(z)})=\sum_{g\in G_{f}} v_g^{(z)}\otimes p_g$ for any $z\in T_f$.
Since $V$ is a simple right $\k_{\tau_f}^{G_f}$-comodule, it follows from \cite[Lemma 2.1]{Lar71} that $$V=\span\{v^{(z_0)}_{g}\mid g\in G_f \}.$$
 Because of the fact that $W$ is a right $H$-comodule and $\{p_{x^{-1}z}\#(z^{-1}\triangleright f)\mid x\in G, z\in T_f\}$ is linearly independent, we can show that $$\span\{v^{(z_0)}_{g} \otimes f\otimes z\mid g\in G_f, z\in T_f\}\subseteq W.$$ This means that $$\tilde{V}=W$$ and thus $\tilde{V}$ is simple.
  \item [(2)]
  For any $f\in F$, let $$C_f=\span\{p_g\# (g^\prime \triangleright f)\mid g, g^\prime\in G\}.$$ According to Lemma \ref{lem:Cf}, we know that $C_f$ is a cosemisimple subcoalgebra of $H$. Suppose $V$ is a simple right $\k_{\tau_f}^{G_f}$-comodule and $\tilde{V}$ be the corresponding simple right $H$-comodule. It is clear that $\tilde{V}$ is a right $C_f$-comodule. Let $\operatorname{cf}(\tilde{V})$ be the coefficient coalgebra of $\tilde{V}$. Then $\operatorname{cf}(\tilde{V})$ is a simple subcoalgebra of $C_f$ and $$\dim_\k(\operatorname{cf}(\tilde{V}))=\dim_\k(\tilde{V})^2.$$
  Suppose $V$ and $W$ are nonisomorphic simple right $\k_{\tau_f}^{G_f}$-comodule, we claim that $\tilde{V}$ and $\tilde{W}$ are nonisomorphic as $H$-comodules. In fact, if $$\alpha :\tilde{V}\rightarrow \tilde{W}$$ is a right $H$-comodule isomorphism, then we have
    \begin{eqnarray*}
    &&\tilde{\rho}(\alpha(v\otimes f\otimes z))\\
    &=&(\alpha\otimes \id)(\tilde{\rho}(v\otimes f\otimes z)) \\
    &=& \sum_{x\in G} \tau(z_x^{-1},g_x^{-1};f)^{-1}\tau(z_x^{-1}g_x^{-1}z, z^{-1};f)\alpha(v_{g_x^{-1}}\otimes f\otimes z_x) \otimes p_{x^{-1}z}\#(z^{-1}\triangleright f),
    \end{eqnarray*}
  for any $v\in V, z\in T_f.$
  For any $v\in V$ and $z\in T_f$, we may suppose $$\alpha(v\otimes f\otimes z)=\sum_{i\in I} w^{(i)}\otimes f\otimes z,$$
  where $w^{(i)}\in W.$
  Since $\{p_{x^{-1}z}\#(z^{-1}\triangleright f)\mid x\in G, z\in T_f\}$ is linearly independent, it follows that
  \begin{eqnarray*}
 && \sum_{i\in I}\sum_{x\in G_f} \tau(1_G,g_x^{-1};f)^{-1}\tau(x^{-1}z, z^{-1};f)w^{(i)}_{x^{-1}}\otimes f\otimes 1_G \otimes p_{x^{-1}z}\#(z^{-1}\triangleright f)\\
  &=&\sum_{x\in G_f} \tau(1_G,x^{-1};f)^{-1}\tau(x^{-1}z, z^{-1};f)\alpha(v_{x^{-1}}\otimes f\otimes 1_G) \otimes p_{x^{-1}z}\#(z^{-1}\triangleright f).
  \end{eqnarray*}
  It turns out that $$\alpha(v_{x^{-1}}\otimes f\otimes 1_G)=\sum_{i\in I}w^{(i)}_{x^{-1}}\otimes f\otimes 1_G,$$
  for any $x\in G_f.$
  Define \begin{eqnarray*}
\alpha^\prime: V&\rightarrow& W\\
v&\mapsto& \sum_{i\in I}w^{(i)}.
\end{eqnarray*}
It is straightforward to show that $\alpha^\prime$ is a right $\k_{\tau_f}^{G_f}$-comodule isomorphism, which is a contradiction.

  Let $\mathcal{O}$ be the set of nonisomorphic simple right $\k_{\tau_f}^{G_f}$-comodules.
      Now taking the sum over all $V\in \mathcal{O}$, we obtain
      \begin{eqnarray*}
      \sum_{V\in \mathcal{O}}(\dim_{\k}(\tilde{V}))^2
      &=&\sum_{V\in \mathcal{O}}(\mid T_f \mid\dim_{\k}(V))^2\\
      &=&\mid T_f \mid^2 \sum_{V\in \mathcal{O}}(\dim_{\k}(V))^2\\
      &=&\mid T_f \mid^2 \mid G_f\mid\\
      &=&\mid T_f \mid \mid G\mid\\
      &=&\dim_{\k}(C_f).
      \end{eqnarray*}
      This indicates that any simple right $C_f$-comodule must be among the $\tilde{V}$. Let $F_G$ be the set of representative elements of disjoint orbits in $F$ under the action of $G$. Observe that $$H=\oplus_{f\in F_G} C_f .$$
      Consequently, every simple right $H$-comodule is isomorphic to $\tilde{V}$ for some simple right $\k_{\tau_f}^{G_f}$-comodule, where $f$ ranges over a choice of one element in each $G$-orbit of $F$.
\end{itemize}
\end{proof}
As a result, we have
\begin{corollary}\label{coro:equi}
Fix a $G$-orbit $O_f$ of $F$. The association $V\mapsto \tilde{V}$ in Theorem \ref{thm:simple} is the object map of an equivalence of categories
$$
\k_{\tau_f}^{G_f}\text{-Comod} \rightarrow C_f\text{-Comod}.
$$
This categorical equivalence induces an isomorphism between the corresponding Grothendieck group $\operatorname{Gr}(\k_{\tau_f}^{G_f}\text{-Comod}) $ and $\operatorname{Gr}(C_f\text{-Comod})$.
\end{corollary}
\begin{remark}\rm
In fact, the proof of Lemma \ref{lem:inducedcomodule}, Theorem \ref{thm:simple} and Corollary \ref{coro:equi} do not rely on the fact that $\k^G{}^\tau\#_{\sigma}\k F$ is a Hopf algebra. In other words, these results still hold for the crossed coproduct coalgebra $\k^G{}^\tau\#\k F$.
\end{remark}
Finally, we consider the Drinfeld double of $\k^G$ for some finite group $G$, which is a special case of $\k^G{}^\tau\#_{\sigma}\k F$.
\begin{example}\rm
Let $G$ be a finite group. Recall that the Drinfeld double $D(\k^G)$ of $\k^G$ has $(\k^G)^{* cop}\otimes \k^G$ as its underlying vector space. Multiplication is given by
\begin{eqnarray*}
(g\otimes p_x)(h\otimes p_y)=\delta_{h^{-1}xh, y}gh\otimes p_y,
\end{eqnarray*}
for all $g, h, x, y\in G$, with identity element $1_G\otimes (\sum_{g\in G}p_g).$
According to \cite[Proposition 10.3.14]{Mon93}, the coalgebra structure of $D(\k^G)^*$ is given by
$$\Delta(p_g\otimes h)=\sum_{x\in G}p_{gx^{-1}}\otimes (x h x^{-1})\otimes p_x\otimes h,$$
and $$\varepsilon(p_g \otimes h)=\delta_{g, 1_G},$$
for any $g, h\in G.$

Suppose that $G$ acts on itself by conjugation and the co-cocycle $\tau$ is trivial.
It is easy to find that $$D(\k^G)^*\cong \k^G {}^\tau\# \k G$$ as coalgebras. An application of Theorem \ref{thm:simple} to this special case gives all the simple right comodules over $D(\k^G)^*.$
Moreover, in such a case, we have $$O_{g}=\{xgx^{-1}\mid x\in G\},$$the conjugacy class of $g$ in $G$,
and $$G_g=Z_g,$$ the centralizer of $g$ in $G.$
\end{example}

\section{The simple subcoalgebras of $\Bbbk^G{}^\tau\#_{\sigma}\Bbbk F$}\label{section5}
With the notations in Proposition \ref{prop:Hopf}, let $H=\k^G{}^\tau\#_{\sigma}\k F$ be the Hopf algebra with the described algebra and coalgebra structures. For convenience, we will use the notations in Subsection \ref{subsection2.3} and fix our notations as in Lemma \ref{lem:inducedcomodule}.

In this section, we attempt to obtain the Grothendieck ring of the category of finite-dimensional right $\k^G{}^\tau\#_{\sigma}\k F$-comodules.
Before that, let us determine the simple subcoalgebras of $\k^G{}^\tau\#_{\sigma}\k F$, that is,
the coefficient coalgebra of simple $\k^G{}^\tau\#_{\sigma}\k F$-comodules.

In the following part, we fix an element $f\in F.$
Let $V=\span\{v^{(1)}, \cdots, v^{(m)}\}$ be a simple right $\k_{\tau_f}^{G_f}$-comodule. By \cite[Lemma 1.2]{Lar71}, there exists a basis $\{c_{ij}\mid 1\leq i, j\leq m\}$ of $\operatorname{cf}(V)$ such that $$\rho(v^{(i)})=\sum_{j=1}^m v^{(j)}\otimes c_{ji},$$
where $(c_{ij})_{m\times m}$ is a multiplicative matrix.
Suppose $$c_{ij}=\sum_{g\in G_f}a_{ij}^g p_g,$$ where $a_{ij}^g\in \k$ for any $1\leq i, j\leq m$ and $g\in G_f.$ Since $(c_{ij})_{m\times m}$ is a multiplicative matrix, it follows that
$$
a_{ij}^{1_G}=\left\{
\begin{aligned}
1 \;,& ~~~ \text{if} ~~~i=j; \\
0 \;,& ~~~ \text{if} ~~~i\neq j. \\
\end{aligned}
\right.
$$

Now we can prove the following lemma, which characterize the coefficient coalgebra of $\tilde{V}$, where $\tilde{V}=(V\otimes \k f)\Box_{\k^{G_f}{}^\tau\#\k F} H$ is the induced right $H$-comodule as in Lemma \ref{lem:inducedcomodule}.
\begin{lemma}\label{lem:coefficientcoalg}
With the notations above, let $V=\span\{v^{(1)}, \cdots, v^{(m)}\}$ be a simple right $\k_{\tau_f}^{G_f}$-comodule and $\tilde{V}=(V\otimes \k f)\Box_{\k^{G_f}{}^\tau\#\k F} H$ be the induced right $H$-comodule as in Lemma \ref{lem:inducedcomodule}. Suppose for any $1\leq i\leq m$, we have $$\rho(v^{(i)})=\sum_{j=1}^m v^{(j)}\otimes (\sum_{g\in G_f}a_{ji}^g p_g),$$
where $a_{ij}^g\in \k$ for any $1\leq i, j\leq m$ and $g\in G_f.$ Let
\begin{eqnarray*}
B&=&\{\sum_{g\in G_f} \tau((z^{\prime})^{-1},g;f)^{-1}\tau((z^\prime)^{-1}gz, z^{-1};f) a_{ji}^{g}p_{(z^\prime)^{-1}gz}\#(z^{-1}\triangleright f)\\
&&\mid 1\leq i, j\leq m, z, z^\prime\in T_f\}.
 \end{eqnarray*}
 Then $B$ is linearly independent and the coefficient coalgebra of $\tilde{V}$ is a simple subcoalgebra of $C_f$ spanned by the elements of $B$.
\end{lemma}
\begin{proof}
Since $V$ is a right $\k_{\tau_f}^{G_f}$-comodule, it follows that
\begin{eqnarray*}
\rho(v^{(i)})&=&\sum_{g\in G_f} (v^{(i)})_g\otimes p_g\\
&=&\sum_{j=1}^m v^{(j)}\otimes c_{ji}\\
&=&\sum_{j=1}^m v^{(j)}\otimes (\sum_{g\in G_f} a_{ji}^g p_g)\\
&=&\sum_{g\in G_f}(\sum_{j=1}^m a_{ji}^gv^{(j)})\otimes p_g,
\end{eqnarray*}
for any $1\leq i\leq m$. This means that $$(v^{(i)})_g=\sum_{j=1}^m a_{ji}^gv^{(j)}.$$
For any $x\in G$, suppose $x=g_xz_x,$ where $g_x\in G_f$ and $z_x\in T_f$.
According to Lemma \ref{lem:inducedcomodule}, we know that
\begin{eqnarray*}
&&\tilde{\rho}(v^{(i)}\otimes f\otimes z)\\
&=&\sum_{x\in G} \tau(z_x^{-1},g_x^{-1};f)^{-1}\tau(z_x^{-1}g_x^{-1}z, z^{-1};f)(v^{(i)})_{g_x^{-1}}\otimes f\otimes z_x \otimes p_{x^{-1}z}\#(z^{-1}\triangleright f)\\
&=&\sum_{x\in G} \tau(z_x^{-1},g_x^{-1};f)^{-1}\tau(z_x^{-1}g_x^{-1}z, z^{-1};f)(\sum_{j=1}^m a_{ji}^{g_x^{-1}}v^{(j)})\otimes f\otimes z_x \otimes p_{x^{-1}z}\#(z^{-1}\triangleright f)\\
&=&\sum_{g\in G_f}\sum_{z^\prime\in T_f} \tau((z^{\prime})^{-1},g^{-1};f)^{-1}\tau((z^\prime)^{-1}g^{-1}z, z^{-1};f)(\sum_{j=1}^m a_{ji}^{g^{-1}}v^{(j)})\otimes f\otimes z^\prime \otimes p_{(z^\prime)^{-1}g^{-1}z}\\
&&\#(z^{-1}\triangleright f)\\
&=&\sum_{j=1}^m\sum_{z^\prime\in T_f}v^{(j)}\otimes f\otimes z^\prime \otimes(\sum_{g\in G_f} \tau((z^{\prime})^{-1},g^{-1};f)^{-1}\tau((z^\prime)^{-1}g^{-1}z, z^{-1};f) a_{ji}^{g^{-1}}p_{(z^\prime)^{-1}g^{-1}z}\\
&&\#(z^{-1}\triangleright f)),
\end{eqnarray*}
for any $1\leq i\leq m, z\in T_f.$ According to Theorem \ref{thm:simple}, $\tilde{V}$ is a simple right $\k^G{}^\tau\#_{\sigma}\k F$-comodule. Then the coefficient coalgebra $\operatorname{cf}(\tilde{V})$ of $\tilde{V}$ is a simple subcoalgebra of $C_f$. It follows from \cite[Lemma 1.2]{Lar71} that $\operatorname{cf}(\tilde{V})$ is a subcoalgebra of $C_f$ spanned by
$B$ and $B$ is linearly independent.
\end{proof}
Now we turn to mention the irreducible characters of simple right $\k^G{}^\tau\#_{\sigma}\k F$-comodules.
\begin{proposition}\label{prop:irreduciblecharacter}
With the notations in Lemma \ref{lem:inducedcomodule}, let $V=\span\{v^{(1)}, \cdots, v^{(m)}\}$ be a simple right $\k_{\tau_f}^{G_f}$-comodule and $\tilde{V}=(V\otimes \k f)\Box_{\k^{G_f}{}^\tau\#\k F} H$ be the induced right $H$-comodule. Suppose for any $1\leq i\leq m$, we have $$\rho(v^{(i)})=\sum_{j=1}^m v^{(j)}\otimes (\sum_{g\in G_f}a_{ji}^g p_g),$$
where $a_{ij}^g\in \k$ for any $1\leq i, j\leq m$ and $g\in G_f.$ Then the irreducible character $\chi(\tilde{V})$ of $\tilde{V}$ is
$$\sum\limits_{i=1}^m\sum\limits_{z\in T_f}\sum\limits_{g\in G_f}\tau(z^{-1},g;f)^{-1}\tau(z^{-1}gz, z^{-1};f)a_{ii}^{g}p_{z^{-1}gz}\#(z^{-1}\triangleright f).$$
\end{proposition}
\begin{proof}
Using Theorem \ref{thm:simple}, one can show that $\tilde{V}$ is a simple right $H$-comodule and $\{v^{(i)}\otimes f\otimes z\mid 1\leq i\leq m, z\in T_f\}$ is a basis of $\tilde{V}$. It follows from \cite[Lemma 1.2]{Lar71} that the elements of $B$ can be arranged into a basic multiplicative matrix $\A=(a_{ij})_{m\mid T_f\mid\times m\mid T_f\mid}$. According to Lemma \ref{lem:char}, we know that $$\chi(\tilde{V})=\sum\limits_{i=1}^m a_{ii}.$$ Note that we have $$\varepsilon(a_{ij})=\delta_{i, j},$$ for any $1\leq i, j\leq m.$ It remains to find that when
$$
\varepsilon(\sum_{g\in G_f} \tau((z^{\prime})^{-1},g;f)^{-1}\tau((z^\prime)^{-1}gz, z^{-1};f) a_{ji}^{g}p_{(z^\prime)^{-1}gz}\#(z^{-1}\triangleright f))\neq 0.
$$
In fact, we have $$\varepsilon(p_g\# f)=\varepsilon(p_g)\varepsilon(f)=\delta_{g,1_G},$$ where $g\in G, f\in F.$
It should be pointed out that $$(z^\prime)^{-1}gz =1_G$$ if and only if $$z=z^\prime,\;\; g=1_G.$$
Besides, we know that $$
a_{ij}^{1_G}=\left\{
\begin{aligned}
1 \;,& ~~~ \text{if} ~~~i=j; \\
0 \;,& ~~~ \text{if} ~~~i\neq j. \\
\end{aligned}
\right.
$$
Then
$$
\varepsilon(\sum_{g\in G_f} \tau((z^{\prime})^{-1},g;f)^{-1}\tau((z^\prime)^{-1}gz, z^{-1};f) a_{ji}^{g}p_{(z^\prime)^{-1}gz}\#(z^{-1}\triangleright f))\neq 0
$$
if and only if
$$z=z^\prime,\;\;i=j.$$
To conclude, one can get
$$\chi(\tilde{V})=\sum\limits_{i=1}^m\sum\limits_{z\in T_f}\sum\limits_{g\in G_f}\tau(z^{-1},g;f)^{-1}\tau(z^{-1}gz, z^{-1};f)a_{ii}^{g}p_{z^{-1}gz}\#(z^{-1}\triangleright f).$$
\end{proof}

\begin{corollary}\label{coro:special}
With the notations in Lemma \ref{lem:inducedcomodule}, let $V=\span\{v^{(1)}, \cdots, v^{(m)}\}$ be a simple right $\k_{\tau_f}^{G_f}$-comodule and $\tilde{V}=(V\otimes \k f)\Box_{\k^{G_f}{}^\tau\#\k F} H$ be the induced right $H$-comodule. Suppose for any $1\leq i\leq m$, we have $$\rho(v^{(i)})=\sum_{j=1}^m v^{(j)}\otimes (\sum_{g\in G_f}a_{ji}^g p_g),$$
where $a_{ij}^g\in \k$ for any $1\leq i, j\leq m$ and $g\in G_f.$
\begin{itemize}
 \item [(1)]If $G_f=\{1_G\}$, then the coefficient coalgebra of $\tilde{V}$ is $C_f$ and the irreducible character $$\chi(\tilde{V})=\sum_{g\in G}p_1\#(g \triangleright f).$$
  \item [(2)]If $G_f=G$, then the coefficient coalgebra of $\tilde{V}$ is a simple subcoalgebra of $C_f$ spanned by $$\{\sum_{g\in G} a_{ji}^{g}p_{g}\# f\mid 1\leq i, j\leq m\}.$$ Moreover, the irreducible character $$\chi(\tilde{V})=\sum\limits_{i=1}^m\sum\limits_{g\in G}a_{ii}^{g}p_{g}\# f.$$
\end{itemize}
\end{corollary}

\begin{example}\rm
Let $p$ be prime and $G=\mathbb{Z}_p=\{g\mid g^p=1\}$. Suppose that $x$ is a generator of $\mathbb{Z}_p$, that is, the order of $x$ is $p$.
Let $\mathbb{Z}[x]=\{\sum\limits_{i=0}^{p-1}a_i x^i\mid a_i \in \mathbb{Z}\}$ be an additive group. Define group actions $\mathbb{Z}_{p}\xleftarrow{\triangleleft}\mathbb{Z}_{p}\times \mathbb{Z}[x] \xrightarrow{\triangleright}\mathbb{Z}[x]$ on the sets by
$$
g^{k}\triangleright \sum\limits_{i=0}^{p-1}a_i x^i= \sum\limits_{i=0}^{p-1} a_i x^{i+k},\;\; g^{k}\triangleleft \sum\limits_{i=0}^{p-1}a_i x^i =g^{k},
$$
for any $1\leq k\leq 2n$, $\sum\limits_{i=0}^{p-1}a_i x^i\in \mathbb{Z}[x]$. Clearly, $(\mathbb{Z}[x], \mathbb{Z}_p)$ together with group actions $\mathbb{Z}_{p}\xleftarrow{\triangleleft}\mathbb{Z}_{p}\times \mathbb{Z}[x] \xrightarrow{\triangleright}\mathbb{Z}[x]$ on the sets is a matched pair. Consider the case when $\sigma$ and $\tau$ are trivial. In this case, let $H(\mathbb{Z}[x], \mathbb{Z}_p)=\k^{\mathbb{Z}_p}{}^{\tau}\#_{\sigma}\k\mathbb{Z}$ be the Hopf algebra defined in Proposition \ref{prop:Hopf} with the described algebra and coalgebra structures. According to Proposition \ref{prop:cosemisimple}, we know that $H(\mathbb{Z}[x], \mathbb{Z}_p)$ is a cosemisimple Hopf algebra.

Denote $\sum\limits_{i=0}^{p-1}a_i x^i$ by $\alpha(a_0, \cdots, a_{p-1})$, where $a_i\in \mathbb{Z}$ for any $0\leq i\leq p-1$. Let $\gamma_p$ be a $p$-cycle. Then for any $\alpha(a_0, \cdots, a_{p-1})$, the $G$-orbit $$O_{\alpha(a_0, \cdots, a_{p-1})}=\{\sum\limits_{i=0}^{p-1}b_i x^i \mid (b_0, \cdots, b_{p-1})=\gamma_n^i (a_0, \cdots, a_{p-1})\;\text{for some } 0\leq i\leq p-1\}.$$
Note that $\mathbb{Z}_p$ is a simple group, whose only subgroups are the trivial subgroup and itself. It follows that
$$
G_{\alpha(a_0, \cdots, a_{p-1})}=\left\{
\begin{aligned}
G,\;\;~~~ &\text{if} ~~~  a_0=a_1=\cdots=a_{p-1}; \\
\{1\},~~~&otherwise .
\end{aligned}
\right.
$$
For any $\alpha(a_0, \cdots, a_{p-1})$, if there exists some $i\neq j$ such that $a_i\neq a_j$, then $C_{\alpha(a_0, \cdots, a_{p-1})}$ is a simple subcoalgebra (see Lemma \ref{lem:Cf}). For any $\alpha(a_0, \cdots, a_{0})$, it follows from Corollary \ref{coro:special} that $$C_{(a_0, \cdots, a_{0})}\cong \mathbb{Z}_p$$ as coalgebras.
\end{example}

Next we attempt to characterize the Grothendieck ring of the category of finite-dimensional right $\k^G{}^\tau\#_{\sigma}\k F$-comodules.
\begin{proposition}\label{prop:tensorproduct}
With the notations in Lemma \ref{lem:inducedcomodule},
suppose $V$ is a simple right $\k_{\tau_f}^{G_f}$-comodule and $W$ is a simple right $\k_{\tau_{f^\prime}}^{G_{f^\prime}}$-comodule. Let $\tilde{V}$ and $\tilde{W}$ be the induced right $\k^G{}^\tau\#_{\sigma}\k F$-comodules of $V$ and $W$, respectively. Then
for any $f^{\prime\prime}\in O_{f, f^{\prime}}$, there exists a non-empty family $\{U_{f^{\prime\prime}}^{(i)}\}_{i\in {I_{f^{\prime\prime}}}}$ of simple right $\k_{\tau_{f^{\prime\prime}}}^{G_{f^{\prime\prime}}}$-comodules such that $$\chi(\tilde{V})\chi(\tilde{W})=\sum\limits_{i\in I_{f^{\prime\prime}}}\sum\limits_{f^{\prime\prime}\in O_{f, f^{\prime}}} \chi(\tilde{U_{f^{\prime\prime}}^{(i)}}).$$  In other words, we have
$$\tilde{V} \otimes \tilde{W}\cong \bigoplus_{f^{\prime\prime}\in O_{f, f^\prime}}\bigoplus_{i\in I_{f^{\prime \prime}}}\tilde{U_{f^{\prime\prime}}^{(i)}}.$$
\end{proposition}
\begin{proof}
Let $V=\span\{v^{(1)}, \cdots, v^{(m)}\}$ and $W=\span\{w^{(1)}, \cdots, w^{(n)}\}$.
For any $1\leq i\leq m, 1\leq j\leq n$, suppose that $$\rho(v^{(i)})=\sum_{l=1}^m v^{(l)}\otimes (\sum_{g\in G_f}a_{li}^g p_g),$$ and $$\rho(w^{(j)})=\sum_{k=1}^n w^{(k)}\otimes (\sum_{h\in G_{f^\prime}}b_{kj}^{h} p_{h}),$$
where $a_{li}^g, b_{kj}^g\in\k$ for any $1\leq l, i\leq  m, 1\leq k, j\leq n, g\in G_f, h\in G_{f^\prime}.$
According to Lemma \ref{lemma:char=ring}, we know that $\tilde{V} \otimes \tilde{W}$ is determined by $\chi(\tilde{V})\chi(\tilde{W})$.
It is a consequence of Proposition \ref{prop:irreduciblecharacter} that
\begin{eqnarray*}
&&\chi(\tilde{V})\chi(\tilde{W})\\
&=&(\sum\limits_{i=1}^m\sum\limits_{z\in T_f}\sum\limits_{g\in G_f}\tau(z^{-1},g;f)^{-1}\tau(z^{-1}gz, z^{-1};f)a_{ii}^{g}p_{z^{-1}gz}\#(z^{-1}\triangleright f))\\
&&(\sum\limits_{j=1}^n\sum\limits_{y\in T_{f^\prime}}\sum\limits_{h\in G_{f^\prime}}\tau(y^{-1},h;f^\prime)^{-1}\tau(y^{-1}hy, y^{-1};f^\prime)b_{jj}^{h}p_{y^{-1}hy}\#(y^{-1}\triangleright f^{\prime}))\\
&=&\sum\limits_{i=1}^m\sum\limits_{z\in T_f}\sum\limits_{g\in G_f}\sum\limits_{j=1}^n\sum\limits_{y\in T_{f^\prime}}\sum\limits_{h\in G_{f^\prime}}
\tau(z^{-1},g;f)^{-1}\tau(z^{-1}gz, z^{-1};f)\tau(y^{-1},h;f^\prime)^{-1}\\
&&\tau(y^{-1}hy, y^{-1};f^\prime)a_{ii}^{g}b_{jj}^{h}\delta_{(z^{-1}gz)\triangleleft (z^{-1}\triangleright f), y^{-1}hy} \sigma(z^{-1}gz; f, f^{\prime})\\
&&p_{z^{-1}gz}\# (z^{-1}\triangleright f)(y^{-1}\triangleright f^{\prime}).
\end{eqnarray*}
Using Lemma \ref{lem:multiunion}, one can show that there exists a subset $O_{f, f^\prime}$ of $F$ such that
\begin{eqnarray*}
\{(z^{-1}\triangleright f)(y^{-1}\triangleright f^{\prime}) \mid z\in T_f, y\in T_{f^{\prime}}\}=
\cup_{f^{\prime\prime}\in O_{f, f^\prime}} O_{f^{\prime\prime}}
\end{eqnarray*}
and the union is disjoint.
Since $\k^G{}^\tau\#_{\sigma}\k F$ is cosemisimple (see Proposition \ref{prop:cosemisimple}), it follows that $\tilde{V} \otimes \tilde{W}$ is a semisimple right $\k^G{}^\tau\#_{\sigma}\k F$-comodule.
Combining Lemma \ref{lem:char} and Proposition \ref{prop:irreduciblecharacter}, we know that
for any $f^{\prime\prime}\in O_{f, f^{\prime}}$, there exists a non-empty family $\{U_{f^{\prime\prime}}^{(i)}\}_{i\in {I_{f^{\prime\prime}}}}$ of simple right $\k_{\tau_{f^{\prime\prime}}}^{G_{f^{\prime\prime}}}$-comodules such that $\chi(\tilde{V})\chi(\tilde{W})=\sum\limits_{i\in I_{f^{\prime\prime}}}\sum\limits_{f^{\prime\prime}\in O_{f, f^{\prime}}} \chi(\tilde{U_{f^{\prime\prime}}^{(i)}})$.
\end{proof}
Next we consider the dual comodule of the simple right $\k^G{}^\tau\#_{\sigma}\k F$-comodule $\tilde{V}$.
\begin{proposition}\label{prop:dual}
With the notations in Lemma \ref{lem:inducedcomodule}, let $V=\span\{v^{(1)}, \cdots, v^{(m)}\}$ be a simple right $\k_{\tau_f}^{G_f}$-comodule and $\tilde{V}=(V\otimes \k f)\Box_{\k^{G_f}{}^\tau\#\k F} H$ be the induced right $H$-comodule. Suppose for any $1\leq i\leq m$, we have $$\rho(v^{(i)})=\sum_{j=1}^m v^{(j)}\otimes (\sum_{g\in G_f}a_{ji}^g p_g),$$
where $a_{ij}^g\in \k$ for any $1\leq i, j\leq m$ and $g\in G_f.$ Then there exists a simple right $\k_{\tau_{f^{-1}}}^{G_{f^{-1}}}$-comodule $U=\{u^{(1)}, \cdots, u^{(m)}\}$ such that
 $$\tilde{V}^*\cong (U\otimes \k f^{-1})\Box_{\k^{G_{f^{-1}}}{}^\tau\#\k F} H.$$
Moreover, if for any $1\leq i\leq m,$ we have
$$\rho(u^{(i)})=\sum_{j=1}^m u^{(j)}\otimes (\sum_{g\in G_{f^{-1}}}b_{ji}^g p_g),$$
where $b_{ij}^g\in \k$ for any $1\leq i, j\leq m$ and $g\in G_{f^{-1}}.$ Then \begin{eqnarray*}
&&\sum\limits_{i=1}^m\sum\limits_{z\in T_f}\sum\limits_{g\in G_f}\tau(z^{-1},g^{-1};f)^{-1}\tau(z^{-1}g^{-1}z, z^{-1};f)a_{ii}^{g^{-1}}\sigma(z^{-1}gz; z^{-1}\triangleright f, (z^{-1}\triangleright f)^{-1})^{-1}\\
&&\tau(zgz^{-1}, z^{-1}g^{-1}z; z^{-1}\triangleright f)^{-1}
p_{(z^{-1}\triangleleft f)(g\triangleleft f)(z^{-1}\triangleleft f)^{-1}}\#(z^{-1}\triangleleft f)\triangleright f^{-1}\\
&=&\sum\limits_{i=1}^m\sum\limits_{z\in T_f}\sum\limits_{g\in G_f}\tau((z\triangleleft f)^{-1},g\triangleleft f; f^{-1})^{-1}
\tau((z\triangleleft f)^{-1}(g\triangleleft f)(z\triangleleft f), (z\triangleleft f)^{-1}; f^{-1})\\
&&b_{ii}^{g\triangleleft f}
p_{(z\triangleleft f)^{-1}(g\triangleleft f)((z\triangleleft f)}\#((z\triangleleft f)^{-1}\triangleright f^{-1}).
\end{eqnarray*}
\end{proposition}
\begin{proof}
Let $V=\span\{v^{(1)}, \cdots, v^{(m)}\}$. Suppose for any $1\leq i\leq m$, we have $$\rho(v^{(i)})=\sum_{j=1}^m v^{(j)}\otimes (\sum_{g\in G_f}a_{ji}^g p_g),$$
where $a_{ij}^g\in \k$ for any $1\leq i, j\leq m$ and $g\in G_f.$
Note that we have $$\chi(\tilde{V}^*)=S(\chi(\tilde{V})).$$
Moreover, for any $g\in G_f, z\in T_f$, a simple computation shows that
\begin{eqnarray*}
&&(z^{-1}gz)\triangleleft (z^{-1}\triangleright f)\\
&=&(z^{-1}\triangleleft (gz\triangleright(z^{-1}\triangleright f)))(gz\triangleleft (z^{-1}\triangleright f))\\
&=&(z^{-1}\triangleleft f)(g\triangleleft(z\triangleright (z^{-1}\triangleright f)))(z\triangleleft (z^{-1}\triangleright f))\\
&=&(z^{-1}\triangleleft f)(g\triangleleft f)(z^{-1}\triangleleft f)^{-1}.
\end{eqnarray*}
With the help of Lemma \ref{lem:actioninverse} and Proposition \ref{prop:irreduciblecharacter}, one can show that
\begin{eqnarray*}
&&\chi(\tilde{V}^*)\\&=&S(\sum\limits_{i=1}^m\sum\limits_{z\in T_f}\sum\limits_{g\in G_f}\tau(z^{-1},g;f)^{-1}\tau(z^{-1}gz, z^{-1};f)a_{ii}^{g}p_{z^{-1}gz}\#(z^{-1}\triangleright f))\\
&=&\sum\limits_{i=1}^m\sum\limits_{z\in T_f}\sum\limits_{g\in G_f}\tau(z^{-1},g;f)^{-1}\tau(z^{-1}gz, z^{-1};f)a_{ii}^{g}\sigma(z^{-1}g^{-1}z; z^{-1}\triangleright f, (z^{-1}\triangleright f)^{-1})^{-1}\\
&&\tau(zg^{-1}z^{-1}, z^{-1}gz; z^{-1}\triangleright f)^{-1}
p_{(z^{-1}\triangleleft f)^{-1}(g\triangleleft f)^{-1}(z^{-1}\triangleleft f)}\#(z^{-1}\triangleright f)^{-1}\\
&=&\sum\limits_{i=1}^m\sum\limits_{z\in T_f}\sum\limits_{g\in G_f}\tau(z^{-1},g;f)^{-1}\tau(z^{-1}gz, z^{-1};f)a_{ii}^{g}\sigma(z^{-1}g^{-1}z; z^{-1}\triangleright f, (z^{-1}\triangleright f)^{-1})^{-1}\\
&&\tau(zg^{-1}z^{-1}, z^{-1}gz; z^{-1}\triangleright f)^{-1}
p_{(z^{-1}\triangleleft f)(g\triangleleft f)^{-1}(z^{-1}\triangleleft f)^{-1}}\#(z^{-1}\triangleleft f)\triangleright f^{-1}\\
&=&\sum\limits_{i=1}^m\sum\limits_{z\in T_f}\sum\limits_{g\in G_f}\tau(z^{-1},g^{-1};f)^{-1}\tau(z^{-1}g^{-1}z, z^{-1};f)a_{ii}^{g^{-1}}\sigma(z^{-1}gz; z^{-1}\triangleright f, (z^{-1}\triangleright f)^{-1})^{-1}\\
&&\tau(zgz^{-1}, z^{-1}g^{-1}z; z^{-1}\triangleright f)^{-1}
p_{(z^{-1}\triangleleft f)(g\triangleleft f)(z^{-1}\triangleleft f)^{-1}}\#(z^{-1}\triangleleft f)\triangleright f^{-1}.
\end{eqnarray*}
Using lemma \ref{lem:actioninverse}, we know that $$g\triangleright f=f$$ if and only if $$(g\triangleright f)^{-1}=f^{-1},$$ if and only if
$$(g\triangleleft f)\triangleright f^{-1}=f^{-1}.$$ This means that $$G_{f^{-1}}=\{g\triangleleft f\mid g\in G_f\}.$$
A similar argument shows that $$T_{f^{-1}}=\{z\triangleleft f\mid z\in T_f\}.$$
Clearly, we have $$\mid O_f\mid=\mid O_{f^{-1}}\mid .$$
By Theorem \ref{thm:simple}, there exists a simple right $\k_{\tau_{f^{-1}}}^{G_{f^{-1}}}$-comodule $U=\span\{u^{(1)}, \cdots, u^{(m)}\}$ such that $$\tilde{V}^*\cong (U\otimes \k f^{-1})\Box_{\k^{G_{f^{-1}}}{}^\tau\#\k F} H.$$
Besides, suppose for any $1\leq i\leq m,$ we have
$$\rho(u^{(i)})=\sum_{j=1}^m u^{(j)}\otimes (\sum_{g\in G_{f^{-1}}}b_{ji}^g p_g).$$
 It follows from Lemma \ref{lem:actioninverse} and Proposition \ref{prop:irreduciblecharacter} that
\begin{eqnarray*}
&&\chi((U\otimes \k f^{-1})\Box_{\k^{G_{f^{-1}}}{}^\tau\#\k F} H)\\
&=&\sum\limits_{i=1}^m\sum\limits_{z^{\prime}\in T_{f^{-1}}}\sum\limits_{g^{\prime}\in G_{f^{-1}}}\tau((z^\prime)^{-1},g^\prime;f^{-1})^{-1}\tau((z^\prime)^{-1}g^\prime z^{\prime}, (z^{\prime})^{-1};f^{-1})\\
&&b_{ii}^{g^\prime}p_{(z^\prime)^{-1}g^\prime z^\prime}\#((z^\prime)^{-1}\triangleright f^{-1})\\
&=&\sum\limits_{i=1}^m\sum\limits_{z\in T_f}\sum\limits_{g\in G_f}\tau((z\triangleleft f)^{-1},g\triangleleft f; f^{-1})^{-1}
\tau((z\triangleleft f)^{-1}(g\triangleleft f)(z\triangleleft f), (z\triangleleft f)^{-1}; f^{-1})\\
&&b_{ii}^{g\triangleleft f}
p_{(z\triangleleft f)^{-1}(g\triangleleft f)((z\triangleleft f)}\#((z\triangleleft f)^{-1}\triangleright f^{-1}).
\end{eqnarray*}
By the fact that $$\chi(\tilde{V}^*)=\chi( (U\otimes \k f^{-1})\Box_{\k^{G_{f^{-1}}}{}^\tau\#\k F} H),$$
one have
\begin{eqnarray*}
&&\sum\limits_{i=1}^m\sum\limits_{z\in T_f}\sum\limits_{g\in G_f}\tau(z^{-1},g^{-1};f)^{-1}\tau(z^{-1}g^{-1}z, z^{-1};f)a_{ii}^{g^{-1}}\sigma(z^{-1}gz; z^{-1}\triangleright f, (z^{-1}\triangleright f)^{-1})^{-1}\\
&&\tau(zgz^{-1}, z^{-1}g^{-1}z; z^{-1}\triangleright f)^{-1}
p_{(z^{-1}\triangleleft f)(g\triangleleft f)(z^{-1}\triangleleft f)^{-1}}\#(z^{-1}\triangleleft f)\triangleright f^{-1}\\
&=&\sum\limits_{i=1}^m\sum\limits_{z\in T_f}\sum\limits_{g\in G_f}\tau((z\triangleleft f)^{-1},g\triangleleft f; f^{-1})^{-1}
\tau((z\triangleleft f)^{-1}(g\triangleleft f)(z\triangleleft f), (z\triangleleft f)^{-1}; f^{-1})\\
&&b_{ii}^{g\triangleleft f}
p_{(z\triangleleft f)^{-1}(g\triangleleft f)((z\triangleleft f)}\#((z\triangleleft f)^{-1}\triangleright f^{-1}).
\end{eqnarray*}
\end{proof}

This leads to the following corollary.
\begin{corollary}
With the notations in Lemma \ref{lem:inducedcomodule}, let $V$ be a simple right $\k_{\tau_f}^{G_f}$-comodule and $\tilde{V}=(V\otimes \k f)\Box_{\k^{G_f}{}^\tau\#\k F} H$ be the induced right $H$-comodule. If $\tilde{V}^*\cong \tilde{V}$, then the set $G_{f, f^{-1}}=\{g\in G\mid g\triangleright f=f^{-1}\}\neq \emptyset$.
\end{corollary}
\begin{proof}
Because of the fact that $$\tilde{V}^*\cong \tilde{V}$$ if and only if $$\chi(\tilde{V}^*)=\chi(\tilde{V}),$$
we can show that $$O_{f^{-1}}=O_f.$$
Thus there exist some $g, g^{\prime}\in G$ such that $$g\triangleright f^{-1}=g^\prime\triangleright f,$$
which follows that $$(g^{-1}g^\prime)\triangleright f=f^{-1}.$$
By Lemma \ref{lem:Gff-1}, we have $$G_{f, f^{-1}}=\{g\in G\mid g\triangleright f=f^{-1}\}\neq \emptyset.$$
\end{proof}

\begin{remark}\label{prop:indicator}\rm
The classical Frobenius-Schur theorem for finite groups was generalized to finite-dimensional semisimple Hopf algebras by Linchenko-Montgomery \cite{LM00}. Bichon proved such a theorem for cosemisimple Hopf algebras with involutive antipode \cite{Bic02}.
Let $H=\k^G{}^\tau\#_{\sigma}\k F$ be the Hopf algebra as in Proposition \ref{prop:Hopf}, $T$ be the left integral defined in the proof of Proposition \ref{prop:cosemisimple} and $V$ be any simple right $H$-comodule.
Following \cite{Bic02}, we define the dual version of Frobenius-Schur indicator of $\chi(V)$ to be
$$\nu_2(\chi(V))=\langle T, m_H\Delta_H(\chi(V))\rangle.$$
Combining Corollary \ref{coro:S2=1} and \cite[Theorem 6]{Bic02}, we have $\nu_2(\chi(V))=0, 1$ or $-1$. Moreover,
$\nu_2(\chi(V))\neq 0$ if and only if $V\cong V^*$. The case $\nu_2(\chi(V))=1$ corresponds to the existence of a $H$-colinear symmetric nondegenerate bilinear form on $V$, while the case $\nu_2(\chi(V))=-1$ corresponds to the existence of a $H$-colinear skew-symmetric nondegenerate bilinear form on $V$. For example, let $H(\mathbb{Z}, \mathbb{Z}_2)$ be the Hopf algebra as in Example \ref{example:H(e,f)} and $V$ be any simple right $H(\mathbb{Z}, \mathbb{Z}_2)$-comodule. A simple computation shows that $\nu_2(\chi(V))=1.$ In the future, we will study the Frobenius-Schur theorem for any cosemisimple Hopf algebra.
\end{remark}

\section{Further results about smash products}\label{section6}
In this section, let $H=\k^G{}^\tau\#_{\sigma}\k F$ be the Hopf algebra as in Proposition \ref{prop:Hopf} with $\sigma(f, f^\prime)=1$ and $\tau(g)=1\otimes 1$ for any $f, f^\prime\in F$ and $g\in G$. That is, $H=\k^G\#\k F.$ For convenience, we will use the notations in Subsection \ref{subsection2.3} and fix our notations as in Lemma \ref{lem:inducedcomodule}.

In the following part, suppose $G$ is an abelian group. In such a case, $\k^G$ is a pointed Hopf algebra, which means that any simple right $\k^G$-comodule is $1$-dimensional.

As stated in the previous section, let us fix an element $f\in F$. Suppose $V=\span\{v\}$ is a simple right $\k^{G_f}$-comodule. We have
$$\rho(v)=v\otimes (\sum_{g\in G_f} a^g p_g),$$
where $a^{g}\in \k$ for any $g\in G_f.$ Besides, we know that $\sum_{g\in G_f} a^g p_g$ is a group-like element in $\k^{G_f}$, which follows that
$$
a^{1_{G}}=1,
$$
and
$$
a^{yx}=a^ya^x.
$$

In this case, Lemma \ref{lem:coefficientcoalg} and Proposition \ref{prop:irreduciblecharacter} have a simpler form.
\begin{lemma}\label{coro:smashsimple}
With the notations above, suppose $G$ is an abelian group. Let $V=\span\{v\}$ be a simple right $\k^{G_f}$-comodule and $\tilde{V}$ be the induced right $\k^G\#\k F$-comodule. Suppose that $$\rho(v)=v\otimes (\sum_{g\in G_f} a^g p_g),$$
where $a^{g}\in \k$ for any $g\in G_f.$ Then the coefficient coalgebra of $\tilde{V}$ is the simple subcoalgebra spanned by
$$
B=\{\sum_{g\in G_f}  a^{g}p_{(z^\prime)^{-1}gz}\#(z^{-1}\triangleright f)\mid  z, z^\prime\in T_f\}.
$$
Moreover, the character $\chi(\tilde{V})$ is
 $$
 \sum\limits_{z\in T_f}\sum\limits_{g\in G_f}a^{g}p_{g}\#(z^{-1}\triangleright f).
 $$
\end{lemma}

Now we can determine $\tilde{V}^*$ under some assumption.
\begin{proposition}\label{prop:smashdual}
With the notations above, let $G$ be an abelian group. Suppose that the action of $F$ on $G$ is trivial, that is, $g\triangleleft f=g$ for any $f\in F, g\in G$. Let $V=\span\{v\}$ be a simple right $\k^{G_f}$-comodule and $\tilde{V}$ be the induced right $\k^G\#\k F$-comodule. Let $H^\prime= \k^{G_{f^{-1}}}\#\k F$. Then we have $$\tilde{V}^*\cong (V^*\otimes \k f^{-1})\Box_{H^\prime} (\k^G\#\k F).$$
Moreover, $\tilde{V}\cong \tilde{V}^*$ if and only if $V\cong V^*$ and the set $G_{f, f^{-1}}=\{g\in G\mid g\triangleright f=f^{-1}\}\neq \emptyset$.
\end{proposition}
\begin{proof}
Suppose
$$\rho(v)=v\otimes (\sum_{g\in G_f} a^g p_g),$$
where $a^{g}\in \k$ for any $g\in G_f.$
It follows from Proposition \ref{prop:irreduciblecharacter} that $$
\chi(\tilde{V})= \sum\limits_{z\in T_f}\sum\limits_{g\in G_f}a^{g}p_{g}\#(z^{-1}\triangleright f).
 $$
We know that
\begin{eqnarray*}
\chi(\tilde{V}^*)
&=&S(\sum\limits_{z\in T_f}\sum\limits_{g\in G_f}a^{g}p_{g}\#(z^{-1}\triangleright f))\\
&=&\sum\limits_{z\in T_f}\sum\limits_{g\in G_f}a^{g} p_{g^{-1}}\# (z^{-1}\triangleright f)^{-1}.
\end{eqnarray*}
Meanwhile, let $V^*=\span\{v^*\}$ be the dual space of $V$. It is clear that $V^*$ is a simple right $\k^G$-comodule with comodule structure
$$
\rho(v^*)=v^*\otimes (\sum_{g\in G_f} a^g p_{g^{-1}}).
$$
An argument similar to the one used in the proof of Proposition \ref{prop:dual} shows that $$G_f=G_{f^{-1}},$$ and $$T_f=T_{f^{-1}}.$$
We find that
\begin{eqnarray*}
\chi((V^*\otimes \k f^{-1})\Box_{H^\prime} (\k^G\#\k F))&=&\sum\limits_{z\in T_{f}}\sum\limits_{g\in G_{f}}a^{g}p_{g^{-1}}\#(z^{-1}\triangleright f^{-1})\\
&=&\chi(\tilde{V}^*).
\end{eqnarray*}
Thus we have  $$\tilde{V}^*\cong (V^*\otimes \k f^{-1})\Box_{H^\prime} (\k^G\#\k F).$$
Moreover, we know that
$$
\chi(\tilde{V})=\chi(\tilde{V}^*)
$$
is equivalent to $$a^{g^{-1}}=a^g,\;\;O_f=O_{f^{-1}}.$$
It is straightforward to show that $$a^{g^{-1}}=a^g$$ if and only if $$V\cong V^*.$$
Besides, it follows from Lemma \ref{lem:Gff-1} that $$O_f=O_{f^{-1}}$$ if and only if $$G_{f, f^{-1}}=\{g\in G\mid g\triangleright f=f^{-1}\}\neq \emptyset.$$
\end{proof}

Suppose that $V=\{v\}$ is a simple right $\k^G$-comodule with comodule structure  $$\rho(v)=\sum_{g\in G} v_g\otimes p_g.$$
For any $f\in F$, we know that $V$ becomes a simple right $\k^{G_f}$-comodule via $$\overline{\rho}(v)=\sum_{g\in G_f}v_g\otimes p_g.$$ Denote the right $\k^{G_f}$-comodule $(V, \overline{\rho})$ by $\overline{V}_f.$
\begin{proposition}\label{prop:multiplication}
With the notations above, let $G$ be an abelian group. Suppose that the action of $F$ on $G$ is trivial, that is, $g\triangleleft f=g$ for any $f\in F, g\in G$. For any $1_F\neq f, f^\prime\in F$, suppose $G_f= G_{f^\prime}$. For any $x, y \in F$, suppose $O_xO_{y}=\{f_1f_2\mid f_1\in O_x, f_2\in O_{y}\}=\cup_{e\in O_{x, y}} O_{e}$ and the union is disjoint. Let $V$ be any simple right $\k^{G_x}$-comodule and $W$ be any simple right $\k^{G_{y}}$-comodule. Suppose $\tilde{V}$ and $\tilde{W}$ are the induced right $\k^G\#\k F$-comodules of $V$ and $W$, respectively.
\begin{itemize}
  \item [(1)]If $x=y=1_F$ and $V\otimes W\cong U$ as right $\k^{G}$-comodules, then $$\tilde{V}\otimes \tilde{W}\cong (U\otimes \k 1_F)\Box_{\k^{G}\#\k F}(\k^G\#\k F).$$
  \item [(2)]If $x=1_F$, $y\neq 1_F$ and $\overline{V}_y\otimes W\cong U$ as right $\k^{G_y}$-comodules, then $$\tilde{V}\otimes \tilde{W}\cong (U\otimes \k y)\Box_{\k^{G_{y}}\#\k F}(\k^G\#\k F).$$
  \item [(3)]If $x\neq1_F$, $y= 1_F$ and $V\otimes \overline{W}_x\cong U$ as right $\k^{G_x}$-comodules, then $$\tilde{V}\otimes \tilde{W}\cong (U\otimes \k x)\Box_{\k^{G_{x}}\#\k F}(\k^G\#\k F).$$
  \item [(4)]If $x, y \neq 1_F$, $1_F \notin O_{x, y}$ and $V\otimes W\cong U$ as right $\k^{G_x}$-comodules. Then $$\tilde{V}\otimes \tilde{W}\cong \bigoplus_{e\in O_{x, y}}(U\otimes \k e)\Box_{\k^{G_{x}}\#\k F}(\k^G\#\k F).$$
  \item [(5)]If $x, y \neq 1_F$, $1_F \in O_{x, y}$ and $V\otimes W\cong U$ as right $\k^{G_x}$-comodules. Suppose for $u\in U$, we have $$\rho(u)=u\otimes (\sum_{g\in G_f} c^g p_g),$$ where $c^g\in\k$ for any $g\in G_f.$ Then there exits a family $\{U^{(i)}\mid 1\leq i\leq \mid O_x\mid\}$ of simple right $\k^{G}$-comodules such that $\overline{U^{(i)}}_x=U$ and $$\sum_{i=1}^{\mid O_x\mid}\chi((U_i\otimes \k 1_F)\Box_{\k^{G}\#\k F}(\k^G\#\k F))=\mid O_x\mid \sum\limits_{g\in G_f}c^{g}p_{g}\#1_F.$$
      Moreover, $\tilde{V}\otimes \tilde{W}$ is isomorphic to
      \begin{eqnarray*}
      &&(\bigoplus_{ i=1}^{\mid O_x\mid}(U_i\otimes \k 1_F)\Box_{\k^{G}\#\k F}(\k^G\#\k F)) \oplus  (\bigoplus_{ e\in O_{x, y}\backslash \{1_F\}}(U\otimes \k e)\Box_{\k^{G_{x}}\#\k F}(\k^G\#\k F)).
      \end{eqnarray*}
\end{itemize}
\end{proposition}
\begin{proof}
We shall adopt the same procedure as in the proof of Proposition \ref{prop:tensorproduct} and only prove $(4)$ and $(5)$. The proofs of $(1)$-$(3)$ is similar to that of $(4)$ and $(5)$.
\begin{itemize}
\item [(4)]For any $v\in V$ and $w\in W$, suppose $$\rho(v)= v\otimes (\sum_{g\in G} a^g p_g)$$ and
$$\rho(w)=w\otimes (\sum_{g\in G} b^g p_g).$$
It follows that $$\rho(v\otimes w)=v\otimes w \otimes (\sum_{g\in G}a^gb^g p_g).$$
According to Lemma \ref{coro:smashsimple}, one can get $$\chi(\tilde{V})=\sum\limits_{r\in T_x}\sum\limits_{g\in G_x}a^{g}p_{g}\#(r^{-1}\triangleright x)$$ and $$\chi(\tilde{W})=\sum\limits_{s\in T_x}\sum\limits_{g\in G_{x}}b^{g}p_{g}\#(s^{-1}\triangleright y).$$
A simple computation shows that
\begin{eqnarray*}
&&\chi(\tilde{V})\chi(\tilde{W})\\
&=&\sum\limits_{r, s\in T_x}\sum\limits_{g, h\in G_x} a^{g}b^{h}(p_{g}\#(r^{-1}\triangleright x))(p_{h}\#(s^{-1}\triangleright {y}))\\
&=&\sum\limits_{r, s\in T_x}\sum\limits_{g, h\in G_x}a^{g}b^{h}\delta_{g, h}p_{g} \#(r^{-1}\triangleright x)(s^{-1}\triangleright y)\\
&=&\sum\limits_{g\in G_x}a^{g}b^{g}p_{g} \#(\sum\limits_{r, s\in T_{x}}(r^{-1}\triangleright x)(s^{-1}\triangleright y)).
\end{eqnarray*}
Meanwhile, we know that
\begin{eqnarray*}
&&\sum_{e \in O_{x, y}}\chi((U\otimes \k e)\Box_{\k^{G_{x}}\#\k F}(\k^G\#\k F))\\
&=&\sum\limits_{g\in G_x}a^{g}b^{g}p_{g}\#(\sum_{e\in O_{x, y}}\sum\limits_{z\in T_x}(z^{-1}\triangleright e)).
\end{eqnarray*}
One can show that both $$\sum\limits_{r, s\in T_{x}}(r^{-1}\triangleright x)(s^{-1}\triangleright y)$$
and $$\sum_{e\in O_{x, y}}\sum\limits_{z\in T_{x}}(z^{-1}\triangleright e)$$ equal to
$$\sum_{t\in O_fO_{f^\prime}} t.$$
Thus
$$\tilde{V}\otimes \tilde{W}\cong \bigoplus_{e\in O_{x, y}}(U\otimes \k e)\Box_{\k^{G_{x}}\#\k F}(\k^G\#\k F).$$
  \item [(5)]With the notations in the proof of $(4)$, if $$(r_0^{-1}\triangleright x)(s_0^{-1}\triangleright y)=1_F,$$ for some $r_0, s_0\in T_f.$
  Then according to Lemma \ref{lem:actioninverse}, we have \begin{eqnarray*}
  r_0^{-1}\triangleright x&=&(s_0^{-1}\triangleright y)^{-1}\\
  &=& s_0^{-1}\triangleright y^{-1}.
  \end{eqnarray*}
  Without loss of generality, suppose that $x=y^{-1}.$
  It follows that
  \begin{eqnarray*}
&&\chi(\tilde{V})\chi(\tilde{W})\\
&=&\sum\limits_{g\in G_x}a^{g}b^{g}p_{g} \#(\sum\limits_{r, s\in T_{x}, r\neq s }(r^{-1}\triangleright x)(s^{-1}\triangleright y))+\mid O_x\mid(\sum\limits_{g\in G_x}a^{g}b^{g}p_{g}\#1_F).
\end{eqnarray*}
Using the same argument as in the proof of Proposition \ref{prop:tensorproduct}, we can easily carry out the proof of this proposition.
\end{itemize}
\end{proof}

Finally, let us generalize Example \ref{example:H(e,f)} and give some new examples.
\begin{example}\rm
Let $G=\mathbb{Z}_{2n}=\{g\mid g^{2n}=1\}$ for some $n\geq 1$. Define group actions $\mathbb{Z}_{2n}\xleftarrow{\triangleleft}\mathbb{Z}_{2n}\times \mathbb{Z} \xrightarrow{\triangleright}\mathbb{Z}$ on the sets by
$$
g^{i}\triangleright j=(-1)^i j,\;\; g^{i}\triangleleft j=g^{i},
$$
for any $1\leq i\leq 2n$, $j\in \mathbb{Z}$. Clearly, $(\mathbb{Z}, \mathbb{Z}_{2n})$ together with group actions $\mathbb{Z}_{2n}\xleftarrow{\triangleleft}\mathbb{Z}_{2n}\times \mathbb{Z} \xrightarrow{\triangleright}\mathbb{Z}$ on the sets is a matched pair. Consider the case when $\sigma$ and $\tau$ are trivial, that is, $\sigma(i, j)=1$ and $\tau(x)=1\otimes 1$ for any $i, j\in \mathbb{Z}$ and $x\in \mathbb{Z}_{2n}$. In such a case, let $H(\mathbb{Z}, \mathbb{Z}_{2n})=\k^{\mathbb{Z}_2}{}^{\tau}\#_{\sigma}\k\mathbb{Z}$ be the Hopf algebra defined in Proposition \ref{prop:Hopf} with the described algebra and coalgebra structures. According to Proposition \ref{prop:cosemisimple}, we know that $H(\mathbb{Z}, \mathbb{Z}_{2n})$ is a cosemisimple Hopf algebra.

If $j=0$, it is a consequence of Lemma \ref{lem:Cf} that $C_0\cong \mathbb{Z}_{2n}$, which means that the simple subcoalgebras of $C_0$ are $1$-dimensional.

For any $j\in \mathbb{Z}_+$, we know that $$C_j=\span\{p_{g^r}\# j\mid 1\leq r \leq 2n\}\oplus \span\{p_{g^r}\# (-j)\mid 1\leq r \leq 2n\},$$ and
$$ \mathcal{C}_j=
\left(\begin{array}{ccccc}
p_1\#(-j)&p_{g^{2n-1}}\# j& p_{g^{2n-2}}\#(-j)&\cdots&p_g\#j\\
p_g\#(-j)&p_1\# j& p_{g^{2n-1}}\#(-j)&\cdots&p_{g^2}\#j\\
p_{g^2}\#(-j)&p_g\#j&p_1\#(-j)&\cdots &p_{g^3}\# j\\
\vdots&\vdots&\vdots&&\vdots\\
p_{g^{2n-1}}\# (-j)&p_{g^{2n-2}}\# j& p_{g^{2n-3}}\# (-j)&\cdots& p_1\#j
 \end{array}\right)
$$
is a multiplicative matrix of $C_j$. It follows from Lemma \ref{lem:Cf} that $C_j$ is s simple subcoalgebra of $H(\mathbb{Z}, \mathbb{Z}_{2n})$ if and only if $n=1.$
Moreover, we have $$G_j=\{g^{r}\mid r=2s\text{ for some }s\in\mathbb{Z}\}$$ and $$T_j=\{1, g\}$$ for any $j\in \mathbb{Z}_+$.
Let $q$ be a $2n$-th primitive root of unit and $w=q^2$. Note that for any $0\leq k\leq n-1$, $\sum\limits_{i=0}^{n-1} w^{ik}p_{g^{2i}}$ is a group-like element in $\k^{G_j}$. It follows that
$$
\k^{G_j}=\bigoplus\limits_{k=0}^{n-1}\k(\sum\limits_{i=0}^{n-1} w^{ik}p_{g^{2i}}).
$$

For any $j\in \mathbb{Z}_+$, $0\leq i\leq n-1$, let $$A_i=
\left(\begin{array}{cc}
p_{g^{2i}}\#(-j)& p_{g^{2i-1}}\# j\\
p_{g^{2i+1}}\#(-j)& p_{g^{2i}}\# j
\end{array}\right).
$$
For any $0\leq k\leq n-1$, denote $E_{j}^{(k)}$ by the subspace of $C_j$ spanned by the entries of $\sum\limits_{i=0}^{n-1}w^{ik}A_{i}$.
Then the set of simple subcoalgebras of $C_j$ is $\{E_{j}^{(k)}\mid 0\leq k\leq n-1\}$.
Moreover, we have $$S(E_{j}^{(k)})=E_{j}^{(n-k)},$$
and $$H(\mathbb{Z}, \mathbb{Z}_{2n})\cong \k \mathbb{Z}_{2n}\oplus(\bigoplus_{j\in \mathbb{Z}_+}\bigoplus_{k=0}^{n-1}E_{j}^{(k)}).$$

Let $$\mathcal{S}=\{\k g^i\mid 1\leq i\leq 2n\}\cup \{E_{j}^{(k)}\mid j\in \mathbb{Z}_+, 0\leq k\leq n-1\}.$$
Combining Lemmas \ref{lem:ZS}, \ref{lemma:char=ring} and Proposition \ref{prop:multiplication}, now we can characterize $\mathbb{Z}\mathcal{S}.$
For any $1\leq i, i^\prime\leq 2n$, $ j, j^\prime \in \mathbb{Z}_+$ and $0\leq k, l\leq n-1$, we have
\begin{eqnarray*}
&&\k g^i\cdot \k g^{i^\prime}=\k g^{i^\prime}\cdot \k g^i=\k g^{i+i^\prime},\\
&&\k g^i\cdot E_{j}^{(k)}= E_{j}^{(k)}\cdot \k g^i=E_{j}^{i+k},\\
&&E_{j}^k\cdot E_{j^\prime}^l=E_{j^\prime}^l\cdot E_{j}^k=\left\{
\begin{aligned}
\k g^{k+l}+\k g^{n+k+l}+E_{2j}^{(k+l)},~~~ &\text{if} ~~~  j=j^\prime; \\
E_{j+j^\prime}^{(k+l)}+E_{\mid j-j^\prime\mid}^{(k+l)}\;\;\;\;\;\;\;\;,~~~&otherwise .
\end{aligned}
\right.
\end{eqnarray*}
\end{example}

\section{Acknowledgement}
The second author was supported by National Key R$\&$D Program of China 2024YFA1013802 and NSFC 12271243. The third author was supported by NSFC 12401040.

\end{document}